\documentclass{amsart}

\usepackage[margin=1.5in]{geometry}

\allowdisplaybreaks[1]

\usepackage{breakurl}

\usepackage{color}
\usepackage{float}
\usepackage{tikz-cd}
\usetikzlibrary{graphs,quotes} 
\usetikzlibrary{arrows.meta,decorations.pathmorphing,decorations.pathreplacing,decorations.shapes}
\usepackage{amsmath,amssymb}
\usepackage{graphicx}
\usepackage{mathrsfs}
\usepackage{mathabx}
\usepackage[all]{xy}
\usepackage{adjustbox}
\usepackage[normalem]{ulem}
\usepackage{spectralsequences}
\usepackage{tcolorbox}
\usepackage{xfrac}
\usepackage{soul}
\usepackage{varwidth}

\usepackage{hyperref}
\usepackage[capitalise, noabbrev]{cleveref}
\usepackage[T1]{fontenc}
\usepackage[utf8]{inputenc}
\usepackage{caption}
\usepackage{subcaption}
\usepackage{multicol}

\newtheorem{theorem}{Theorem}[section]
\newtheorem{cor}[theorem]{Corollary}

\newtheorem{proposition}[theorem]{Proposition}

\newtheorem{remark}[theorem]{Remark}

\newtheorem*{thm*}{Theorem}
\newtheorem*{cor*}{Corollary}
\newtheorem*{lem*}{Lemma}
\newtheorem*{prop*}{Proposition}

\theoremstyle{definition}
\newtheorem{example}[theorem]{Example}
\newtheorem{deff}[theorem]{Definition}

\theoremstyle{plain}

\numberwithin{theorem}{section}
\numberwithin{equation}{section}
\numberwithin{figure}{section}
\DeclareMathOperator{\Ext}{Ext}

\DeclareMathOperator{\Sq}{Sq}

\DeclareMathOperator{\out}{out}
\DeclareMathOperator{\inn}{in}

\newcommand{\mF}{{\mathbb F}}

\newcommand{\mM}{{\mathbb M}}

\newcommand{\mR}{{\mathbb R}}

\newcommand{\umF}{\underline{\mF}}

\newcommand{\0}{\langle 0 \rangle }

\newcommand{\bigslant}[2]{{\raisebox{.2em}{$#1$}\left/\raisebox{-.2em}{$#2$}\right.}} 

\definecolor{green'}{HTML}{009E73}
\definecolor{pink'}{HTML}{CC79A7}
\definecolor{orange'}{HTML}{D55E00}
\definecolor{darkblue'}{HTML}{0072B2}
\definecolor{yellow'}{HTML}{F0E442}
\definecolor{lightblue'}{HTML}{56B4E9}
\definecolor{gold'}{HTML}{E69F00}

\usepackage[style=numeric,maxnames=99]{biblatex}
\title{Graphs arising from the dual Steenrod algebra}
\author{Connor Elliott, Courtney Hauf, Kai Morton, Sarah Petersen, Leticia Schow}

\begin{document}

\begin{abstract}
We extend Wood's graph theoretic interpretation of certain quotients of the mod $2$ dual Steenrod algebra to quotients of the mod $p$ dual Steenrod algebra where $p$ is an odd prime and to quotients of the $C_2$-equivariant dual Steenrod algebra. We establish connectedness criteria for graphs associated to monomials in these algebra quotients and investigate questions about trees and Hamilton cycles in these settings. We also give graph theoretic interpretations of algebraic structures such as the coproduct and antipode arising from the Hopf algebra structure on the mod $p$ dual Steenrod algebra and the Hopf algebroid structure of the $C_2$-equivariant dual Steenrod algebra.

\end{abstract}
\maketitle
\tableofcontents

\section{Introduction}
The action of the mod $p$ Steenrod algebra $A$ and its dual ${A}^\vee$ on mod $p$ cohomology and homology, respectively, plays a foundational role in homotopy theory computations. In particular, it serves as input to the Adams spectral sequence \cite{Adams58}, which is one of the most powerful tools for computing stable homotopy groups. A primary example of the Adams spectral sequence's use includes Hill-Hopkins-Ravenel's solution of the Kervaire invariant one problem and Lin-Wang-Xu's solution of the last Kervaire invariant one problem \cite{HillHopkinsRavenel16,LinWangXu25}. These solutions used the Adams spectral sequence, and thereby information coming from the action of the dual Steenrod algebra on homology, to answer the question of when a framed $4k + 2$ dimensional manifold can be surgically converted into a sphere. 

Because the dual Steenrod algebra plays a central role in stable homotopy calculations, it is worthwhile to study it from a variety of points of view (see, for example \cite{Smith07,WalkerWood18a,WalkerWood18b,Wood97}). The purpose of this paper is to study the dual Steenrod algebra $A^\vee$ from a graph theory perspective. In particular, we extend a graphical construction due to Wood \cite[\textsection 8]{Wood98} and studied by Yearwood \cite{Yearwood19} and Larson \cite{Larson22} at the prime $2$ to odd primes $p$ and the $C_2$-equivariant setting. For the sake of exposition, we discuss the odd primary and $C_2$-equivariant results separately.

\subsection{Odd primary results} Let $p$ be an odd prime and recall the mod $p$ dual Steenrod algebra $A^\vee$ is a graded Hopf algebra with
\[
A^\vee \cong \mF_p [ \xi_1, \xi_2, \xi_3, \cdots ] \otimes E[ \tau_0, \tau_1, \cdots],
\]
where $| \xi_i | = 2(p^i - 1)$ and $| \tau_i | = 2p^i - 1$. The coproduct on $A^\vee$ is given by
\begin{align*}
    \psi(\xi_n) & = \sum_{i=0}^n \xi_{n-i}^{p^i} \otimes \xi_i  \\
    \psi (\tau_n) & = \tau_n \otimes 1 + \sum_{i=0}^n \xi_{n-i}^{p^i} \otimes \tau_i
\end{align*}
while the unit map $\eta: \mF_p \to A^\vee$ and counit map $\epsilon: A^\vee \to \mF_p$ are isomorphisms in degree $0$. Finally, the  antipode \color{black} $c: A^\vee \to A^\vee$ is defined recursively by 
\begin{align*}
    c(\xi_0) & = 1 &    c(\tau_0) &  = \tau_0 \\
    \sum_{i=0}^{n} \xi_{n-i}^{p^i} c(\xi_i) & = 0 & 
    \tau_n + \sum_{i=0}^{n-1} \xi_{n-i}^{p^i} c(\tau_i) & = 0. \\
\end{align*}

Given an integer $n \geq 0,$ the quotient $A^\vee(n)$ is defined by
\[
A^\vee(n) = \bigslant{A^\vee}{(\xi_1^{p^n}, \xi_2^{p^{n-1}}, \cdots, \xi_n^p, \xi_{n + 1}, \xi_{n + 2}, \cdots, \tau_{n + 1}, \tau_{n + 2}, \cdots)}.
\]
In other words,
\[
A^\vee(n) \cong \bigslant{\mF_p [\xi_1, \xi_2, \cdots, \xi_n, \tau_0, \tau_1, \cdots, \tau_n]}{(\xi_1^{p^n}, \xi_2^{p^{n-1}}, \cdots, \xi_n^p,\tau_0^2, \tau_1^2, \cdots, \tau_n^2)}.
\]

These $A^\vee(n)$ are Hopf algebra quotients so the Hopf algebra structure on the dual Steenrod algebra  $A^\vee$ descends to the quotient $A^\vee(n)$ \cite{Rosen72, AdamsMargolis74}. Dualizing these quotient Hopf algebras yields finite subalgebras $A(n)$ of the Steenrod algebra $ A$ which are generated by reduced $p$-th power operations $P^i$, where $i \leq p^{n-1}$, and the Bockstein operation $\beta$. 

Both the sub-Hopf algebras $A(n)$ and their duals $A^\vee (n)$ facilitate computation of homotopy groups via the Adams spectral sequence. Many Adams spectral sequence computations begin with $\Ext$ over the Steenrod algebra $A$. This initial computation can often be done over the simpler finite sup-Hopf algebra $A(n)$ instead via a change-of-rings isomorphism \cite[A1.3.12]{Ravenel86}. 

One of the main purposes of this paper is to give a graph theoretic interpretation of the quotients $A^\vee (n)$ when $p$ is an odd prime. Similarly to Wood's construction at the prime $2$ \cite[\textsection 8]{Wood98}, this is accomplished by associating to every monomial $x \in A^\vee (n)$ a graph $G_x$ on the vertex set $\{1, 2, 2p, 2p^2, \cdots, 2p^n \}$. We then study the properties of the ensuing graphs. 

We defer an explicit description of the construction of these graphs to \cref{sec:oddprimes} and focus on the statements of the main results here. Taking addition to be represented as disjoint union of graphs, it is sufficient to restrict our study to monic monomials $x \in A^\vee (n)$.

Suppose $x$ is a monic monomial in $A^\vee (n)$. Then
\[
x = \tau_0^{\epsilon_0} \tau_1^{\epsilon_1} \cdots \tau_n^{\epsilon_n} \xi_1^{r_1} \xi_2^{r_2} \cdots \xi_n^{r_n}
\]
where $\epsilon_i \in \{0,1\}$ and $0 \leq r_i \leq p^{n + 1 - i} - 1$. We can rewrite $x$ using the $p$-adic expansion of $r_i$, that is
\[
r_i = \sum_{m=0}^{n-i} c_{i, n-i-m} p^{n-i-m}.
\]
Specifically, we can write
\begin{equation} \label{eq:pSF}
x = \tau_0^{\epsilon_0} \tau_1^{\epsilon_1} \cdots \tau_n^{\epsilon_n} \prod_{i=1}^n \prod_{m=0}^{n-i} \xi_i^{c_{i,n-i-m}p^{n-i-m}}.
\end{equation}
Rewriting the monomial $x$ in this way facilitates extracting a connectedness criterion for the graph $G_x$ from the data of the monomial $x$. We call $x$ as written in (\ref{eq:pSF}) a monic monomial in standard form (Definition \ref{def:StandardFormOdd}). In order to state this connectedness criterion, we first need to introduce some notation.

Let $t$ be an integer with $1 \leq t \leq n + 1$ and define $T_t$ { (Definition \ref{def:Tt})} to be the set of all $(t+1)$-tuples $(b_0, b_1, ,\cdots b_{t - 1}, b_t) $ where $b_i \in \{0, 1, \cdots, n\}$ with no two entries being the same.  Further let $T_t' \subset T_t$ (Definition \ref{def:T't}) be the subset of $T_t$ where $b_0 = r$ and $b_t = s$. \color{black}

\begin{thm*}[\cref{thm:connectedOdd}]
    The graph $G_x$ is connected if and only if 
    \begin{enumerate}
        \item the integers 
        \[
        C(r,s) : = \sum_{t = 1}^n \sum_{T_t'} \prod_{k = 1}^t c_{|b_k-b_{k-1}|, \max (b_{k-1}, b_k)-1} 
        \]
        are positive for all integers $r$ and $s$ with $0 \leq r < s \leq n$,
        \item and the sum 
        \[
        \sum_{k=0}^n \epsilon_k > 0.
        \]
    \end{enumerate}
\end{thm*}

In \cite{Yearwood19}, Yearwood observed that the graphs $G_x$ can also naturally be viewed as directed graphs with edges oriented in the direction of the larger vertex. We will denote the directed version of the graph $G_x$ by $G_x^{\textup{dir}}$. With the directed version of Wood's construction in mind, Larson studied unilaterality, the appropriate connectedness property for directed graphs, in the $p = 2$ setting. In the odd primary setting we  identify exactly when the directed graph $G_x^\textup{dir}$ is unilateral. In order to state our result, we define ${T}_t''$ (Definition \ref{def:T''t}) to be the set of all $(t + 1)$-tuples $(b_0, b_1, \cdots, b_t)$ with $b_i \in \{0, 1, \cdots, n + 1\}$ and $r = b_0 < b_1 < \cdots < b_t = s$. \color{black}
\begin{thm*}[\cref{thm:unilateralOdd}]
    The directed graph $G_x^{\textup{dir}}$ is unilateral if and only if 
    \begin{enumerate}
        \item the integers
        \[
        U(r,s) := \sum_{t=1}^{n}\sum_{{T}''_t}\prod_{k=1}^t c_{b_{k-1},b_k}
        \]
        are positive for all integers $r,s$ with $0 \leq r < s \leq n + 1$,
        \item and furthermore $x$ is divisible by $\tau_0$.
    \end{enumerate}
\end{thm*}

With these connectedness results in hand, we can also identify the graphs $G_x$ that are trees.

\begin{thm*}[\cref{thm:treesOdd}]
    Suppose $x$ is a monic monomial in $A^\vee(n)$ such that $G_x$ is connected. Then $G_x$ is a tree if and only if
   \begin{enumerate}
       \item the sum \[
            \sum_{i=1}^{n} \alpha_p(r_i) = n,
            \]
            where \( \alpha_p(r_i) \) denotes the number of nonzero digits in the base-\(p\) expansion of \( r_i \), 
            \item and exactly one \( \tau_k \) divides the monic monomial \( x \), for some \( 0 \leq k \leq n \).
   \end{enumerate}
\end{thm*}

Among the questions posed by Wood in \cite[\textsection 8]{Wood98} is whether there is an algebraic criterion that describes when the graph $G_x$ has a Hamilton cycle. Larson answered this question affirmatively in the $p = 2$ setting (\cite[Theorem 1.4]{Larson22}). We give an analogue of this result in the odd primary setting.
\begin{thm*}[\cref{thm:oddprimeHamilton}]
    Suppose $x = \tau_0^{\epsilon_0} \cdots \tau_n^{\epsilon_n} \xi_1^{r_1} \cdots \xi_n^{r_n} \in A^\vee(n)$ where $\epsilon_i \in \{0,1\}$ and $0 \leq r_i \leq p^{n + 1 - i} - 1$. The corresponding graph $G_x$ has a Hamilton cycle if
    \begin{enumerate}
        \item $n>0$,
        \item for each $j$ with $0 \leq j \leq n$,
        \[
            f_{\inn}(2p^j) + f_{\out} (2p^j) \geq         \frac{n}{2}, 
        \]
        \item and \[\sum_{k=0}^n  \epsilon_k \geq \frac{n}{2}.\]
    \end{enumerate}  
\end{thm*}

Finally, we generalize the results of \cite{Yearwood19} and \cite{Larson22} giving a graphical interpretation of some of the Hopf algebra structure on $A^\vee(n)$. Specifically we give a graphical interpretation of the coproduct and antipode on elements the $\xi_i^{p^j}$ and $\tau_i \in A^\vee (n)$ when $p$ is an odd prime. 

\begin{thm*}[\cref{thm:coproductOdd}]
\hfill

\begin{enumerate}
    \item[I.] 
    The coproduct $\psi(\xi_i^{p^j}) \in A^\vee(n)$ is the sum of tensors of all pairs of edges that make length $2$ directed paths from the vertex $2 p^j$ to the vertex $2 p^{i +j}$ in the complete graph on the vertices $\{1, 2, 2p, 2p^2, \cdots, 2p^n \}$ considered as a directed graph.
     
    \item[II.] Similarly, coproduct $\psi(\tau_i) \in A^\vee(n)$ is the sum of tensors of all pairs of edges that make length $2$ directed paths from the vertex $1$ to the vertex $2 p^i$.
     \end{enumerate}
\color{black}
\end{thm*}

\begin{thm*}[\cref{thm:AntipodeOdd}]
\hfill

    \begin{enumerate}
        \item[I.] 
        The antipode $c(\xi_i^{p^j}) \in A^\vee(n)$ is the signed sum of all directed paths from the vertex $2 p^j$ to the vertex $2 p^{i + j}$ in the complete graph on the vertices $\{1,2, 2p, 2p^2, \cdots, 2p^n\}$ considered as a directed graph. The positive terms in the sum correspond exactly to paths consisting of an even number of edges from the vertex $2$ to the vertex $2p^{i +j}$.

    \item[II.] Similarly, the antipode $c(\tau_i) \in A^\vee(n)$ is the signed sum of all directed paths from the vertex $1$ to the vertex $2 p^{i}$. The positive terms correspond exactly to paths consisting of an even number of edges from the vertex $1$ to the vertex $2 p^i$.
    \end{enumerate}
    \color{black}
\end{thm*}

As a corollary, we obtain another characterization of the unilaterality of the directed graphs $G_x^\textup{dir}$.

\begin{cor*}[\cref{cor:unilateralityOdd}]
    Suppose $x$ is a monic monomial in $A^\vee (n)$. Then the directed graph $G_x^{\textup{dir}}$ is unilateral if and only if
    \begin{enumerate}
        \item for each $\xi_i^{p^j} \in A^\vee (n)$, at least one summand of $c(\xi_i^{p^i})$ is a factor of $x$
        \item and $\tau_0$ is a factor of $x$.
    \end{enumerate}
\end{cor*}

\subsection{$C_2$-equivariant Results} 
While homotopy theory studies the properties of topological spaces that are preserved under continuous deformations (i.e. homotopies), similar questions can be asked about spaces with rotation and reflection symmetries. This leads to the field of equivariant homotopy theory where a group $G$ acts on the topological spaces under consideration. In this setting, there are also familiar invariants such as homology and homotopy (see \cite{AlaskaNotes96} for a comprehensive introduction). However, their definition is now more complicated. Specifically, instead of being integer graded, these equivariant theories are graded on the additive group underlying the real representation ring $RO(G)$. Intuitively, this comes from the fact that we are now interested in probing our equivariant spaces not only with the ordinary nonequivariant spheres $S^n$ but also with equivariant representation spheres $S^V$, which are formed by taking the one-point compactification of a $G$-representation $V$ and assigning the point at infinity the trivial action. 

Similarly to the nonequivariant setting, equivariant Steenrod algebras act on equivariant (co)homology theories and this serves as input to equivariant Adams spectral sequences. In general, equivariant dual Steenrod algebras are more complicated than their classical nonequivariant counterparts. For instance, the $C_p$-equivariant dual Steenrod algebra is not free over its coefficients \cite{SankarWilson22,HuKrizSombergZou23}. Further complicating matters, we do not currently have, nor expect to have in the near future, complete computations of the $G$-equivariant dual Steenrod algebra when $G$ an arbitrary finite or compact Lie group. Hence in this paper, we focus on the simplest nontrivial case where the group of equivariance is $G = C_2$, the cyclic group of order $2$. In this case, the $C_2$-equivariant dual Steenrod algebra actually has a form remarkably similar to the nonequivariant mod $p$ dual Steenrod algebra. 

In order to introduce the $C_2$-equivariant dual Steenrod algebra, we observe a few representation theoretic preliminaries. The group $C_2$ has two irreducible real representations, the one dimensional trivial representation and the one dimensional sign representation $\sigma$. Throughout, we will let $\rho$ denote the $C_2$-regular representation $1 + \sigma$. This leads to a bigraded theory which we can view as indexed by ${(i-j) + j \sigma}$. Here $i$ is the topological dimension while $j$ is the weight or twisted dimension. And indeed, the $C_2$-equivariant dual Steenrod algebra is bigraded. Moreover, it also has the structure of a Hopf algebroid rather than a Hopf algebra as in the nonequivariant mod $p$ setting. 

To describe the $C_2$-equivariant dual Steenrod algebra as a Hopf algebroid, we must first introduce the bigraded coefficient ring 
\[
\mM_2 \cong \mF_2[a,u] \oplus \frac{\mF_2[a,u]}{(a^\infty, u^\infty)} \{ \theta \},
\]
where $|a|=(-1,-1)$, $|u|=(0,-1)$. A more detailed description of $\mM_2$ can be found in \cref{sec:M2}. The Hopf algebroid $(\mM_2, A_{C_2}^\vee)$ has underlying algebra
\begin{equation} \label{eq:C2dualSteenrod}
A_{C_2}^\vee \cong \mM_2 [ \xi_1, \xi_2, \cdots, \tau_0, \tau_1, \cdots] / (\tau_i^2 = (u + a\tau_0)\xi_{i + 1} + a \tau_{i + 1})
\end{equation}
where $| \xi_i |  = \rho(2^i - 1)$ and $| \tau_i | = 2^i \rho - \sigma$. The coproduct is given by 
\begin{align*}
    \psi(\xi_n) & = \sum_{i=0}^n \xi_{n - i}^{2^i} \otimes \xi_i \\
    \psi(\tau_n) & = \tau_n \otimes 1 + \sum_{i = 0}^n \xi_{n - i}^{2^i} \otimes \tau_i
\end{align*}
(see \cite[Theorem 6.41]{HuKriz01} or \cite[Theorem 2.14]{LiShiWAngXu19} for the full Hopf algebroid structure). 

Given an integer $n \geq 0$, define the quotient $A_{C_2}^\vee (n)$ by
\[
A_{C_2}^\vee (n) = \frac{A_{C_2}^\vee}{(\xi_1^{2^n}, \xi_2^{2^{n-1}}, \cdots, \xi_n^2, \xi_{n + 1}, \xi_{n + 2}, \cdots, \tau_{n + 1}, \tau_{n + 2}, \cdots)}.
\]
In other words,
\[
A_{C_2}^\vee(n) \cong \bigslant{\mathbb{M}_2[\xi_1,\xi_2,..., \xi_n, \tau_0,\tau_1,..., \tau_n]}{(\xi_i^{2^{n-i+1}},\,\tau_i^2=(u+a\tau_0)\xi_{i+1}+a\tau_{i+1})}.
\]

One of the primary purposes of this paper is to give a graph theoretic interpretation of the quotients $A_{C_2}^\vee (n)$. Similarly to Wood's construction at the prime $2$ \cite[\textsection 8]{Wood98}, this is accomplished by associating to every monomial $x \in A_{C_2}^\vee (n)$ a graph $G_x$ on the vertex set $\{\sigma, \rho, 2\rho, 2^2 \rho, \cdots, 2^n \rho \}$. We then study the properties of the ensuing graphs. We defer a description of the construction of these graphs to \cref{sec:C2} and focus on the statements of the main results here. 

Given a monomial $x \in A_{C_2}^\vee (n)$ we can use the relation $\tau_i^2=(u+a\tau_0)\xi_{i+1}+a\tau_{i+1}$ to write $x$ as a finite sum consisting of terms of the form
\begin{equation} \label{eq:C2terms1}
w \tau_0^{\epsilon_0} \tau_1^{\epsilon_1} + \cdots + \tau_n^{\epsilon_n} \xi_1^{r_1} \xi_2^{r_2} \cdots \xi_n^{{r_n}}
\end{equation}
where $w \in \mM_2$, $\epsilon_i \in \{0,1\}$ and $0 \leq r_i \leq 2^{n + 1 - i} - 1$. Because addition is represented as disjoint union in our graphical interpretation, it is sufficient to focus on monic monomials of the form (\ref{eq:C2terms1}). Similarly to Wood's construction in the nonequivariant setting at the prime $2$, we will also utilize the $2$-adic expansion of $r_i$,
\[
r_i = \sum_{m  = 0}^{n - i} c_{i, n - i-m} \cdot 2^{n - i - m}.
\]
Specifically, given a monic monomial of the form (\ref{eq:C2terms1}) we can write 
\begin{equation} \label{eq:C2sf}
x = \tau_0^{\epsilon_0} \tau_1^{\epsilon_1} \cdots \tau_n^{\epsilon_n} \left(\prod_{i = 1}^n \prod_{m = 0}^{n - i} \xi_i^{c_{i, n - i -m} \cdot 2^{n - i - m}} \right).
\end{equation}
We call $x$ written as in (\ref{eq:C2sf}) a monic monomial in standard form (Definition \ref{def:StandardFormEquivariant}). 
Rewriting $x$ in this way again facilitates the extraction of a connectedness criterion for the graph $G_x$ from the data of the monomial $x$. In fact, this connectedness criterion ends up having the same statement as in the nonequivariant $p$ an odd prime case (\cref{thm:connectedOdd}). 
\begin{thm*}[\cref{thm:connectedC2}]
    Suppose $x \in A_{C_2}^\vee(n)$ is a monic monomial { in standard form}. Then the graph $G_x$ is connected if and only if 
    \begin{enumerate}
        \item the integers 
        \[
        C(r,s) : = \sum_{t = 1}^n \sum_{T_t'} \prod_{k = 1}^t c_{|b_k-b_{k-1}|, \max (b_{k-1}, b_k)-1} 
        \]
        are positive for all integers $r$ and $s$ with $0 \leq r < s \leq n$ and
        \item the sum 
        \[
        \sum_{k=0}^n \epsilon_k > 0.
        \]
    \end{enumerate}
\end{thm*}

Similarly to the nonequivariant setting, the graph $G_x$ can also naturally be viewed as a directed graph with edges oriented in the direction of the vertex of larger representation dimension. We will again denote the directed version of the graph $G_x$ by $G^\textup{dir}_x.$ Much like the nonequivariant $p$ an odd prime case, we have a similar criterion for unilaterality. 

\begin{thm*}[\cref{thm:unilateralC2}]
     Suppose $x \in A_{C_2}^\vee(n)$ is a monic monomial { in standard form}. Then the directed graph $G_x^{\text{dir}}$ is unilateral if and only if
    \begin{enumerate}
        \item the integers
        \[
        U(r,s) := \sum_{t=1}^{n}\sum_{{T}''_t}\prod_{k=1}^t c_{b_{k-1},b_k}
        \]
        are positive for all integers $r,s$ with $0 \leq r < s \leq n + 1$, 
        \item and $x$ is divisible by $\tau_0$.
    \end{enumerate}
\end{thm*}

Similarly to the nonequivariant $p$ an odd prime case, we also give an algebraic criterion that describes when the graph $G_x$ has a Hamilton cycle. 

\begin{thm*}[\cref{thm:HamiltonCycleC2}]
    Suppose $x = \tau_0^{\epsilon_0} \cdots \tau_n^{\epsilon_n} \xi_1^{r_1} \cdots \xi_n^{r_n}$ is a monic monomial { in standard form} in $ A_{C_2}^\vee(n)$. The corresponding graph $G_x$ has a Hamilton cycle if
    \begin{enumerate}
        \item $n>0$,
        \item for each $j$ with $0 \leq j \leq n$,
        \[
            \deg_{\inn}(2^j \rho) + \deg_{\out} (2^j \rho) \geq         \frac{n}{2}, 
        \]
        \item and \[\sum_{k=0}^n  \epsilon_k \geq \frac{n}{2}.\]
    \end{enumerate}  
\end{thm*}

Finally, we give a graphical interpretation of some of the Hopf algebroid structure on $A_{C_2}^\vee(n)$. Specifically we give a graphical interpretation of the coproduct and antipode on the elements $\xi_i^{p^j}$ and $\tau_i \in A_{C_2}^\vee (n)$.

\begin{thm*}[\cref{thm:coprodC2}] 
\hfill
\begin{enumerate}

    \item[I.] 
    The coproduct $\psi(\xi_i^{p^j}) \in A^\vee_{C_2}(n)$ is the sum of tensors of all pairs of edges that make length $2$ directed paths from the vertex $2^j \rho $ to the vertex $2^{i + j} \rho$ in the complete graph on the vertices $\{\sigma, \rho, 2 \rho, 2^2 \rho, \cdots 2^n \rho\}$ considered as a directed graph.
    \item[II.] Similarly, the coproduct $\psi(\tau_i) \in A_{C_2}^\vee(n)$ is the sum of tensors of all pairs of edges that make length $2$ directed paths from the vertex $\sigma$ to the vertex $2^i \rho$.
\end{enumerate}
\color{black}
\end{thm*}

\begin{thm*}[\cref{thm:AntipodeC2}]
\hfill 
\begin{enumerate}

    \item[I.] 
    The antipode $c(\xi_i^{p^j}) \in A^\vee_{C_2}(n)$ is the sum of all directed paths from the vertex $2^j \rho$ to the vertex $2^{i + j} \rho$ in the complete graph on the vertices $\{\sigma, \rho, 2 \rho, 2^2 \rho, \cdots 2^n \rho\}$ considered as a directed graph.

    \item[II.] Similarly, the antipode $c(\tau_i) \in A^\vee_{C_2} (n)$ is the sum of all directed paths from the vertex $\sigma$ to the vertex $2^i \rho$.
    \color{black}
\end{enumerate}
\end{thm*}

As a corollary, we also obtain another characterization of the unilaterality of the directed graphs $G_x^\textup{dir}$.

\begin{cor*}[\cref{cor:unilatC2}]
    Suppose $x$ is a monic monomial { in standard form} in $A_{C_2}^\vee (n)$. Then the directed graph $G_x^{\textup{dir}}$ is unilateral if and only if
    \begin{enumerate}
        \item for each $\xi_i^{p^j} \in A_{C_2}^\vee (n)$, at least one summand of $c(\xi_i^{p^i})$ is a factor of $x$
        \item and $\tau_0$ is a factor of $x$.
    \end{enumerate}
\end{cor*}

\subsection{Open Questions} These results suggest some natural questions for further study, some of which partially overlap with questions posed by Larson in \cite[\textsection 5.4]{Larson22}.
\begin{enumerate}
    \item In \cite[\textsection 8]{Wood98}, Wood points out that the mod 2 Steenrod algebra $A$ (as opposed to its dual) and some of its subalgebras $A(n)$ can be interpreted graph theoretically. What would the results of \cite{Yearwood19,Larson22} or this paper look like in that setting?
    \item Given the Hopf algebra structure on $A^\vee(n)$, can one give a graph theoretic interpretation of the coproduct $\psi$ and antipode $c$ on an arbitrary monomial $x \in A^\vee (n)$? Similarly, given the Hopf algebroid structure on $A_{C_2}^\vee(n)$, can one give a graph theoretic interpretation of the coproduct $\psi$ and antipode $c$ on an arbitrary monomial $x \in A^\vee (n)$?
    \item In \cite[\textsection 5]{Wood98} Wood describes two procedures one can perform in the mod 2 Steenrod algebra $A$, called stripping and strapping, that together allow one to derive all of the Adem relations from the single relation $\Sq_1  \Sq_1 = 0$. A step in the process of recovering the Adem relations involves assigning to each monomial $\xi_i^{2^j} \in A^\vee$ a “stripping operator” which is analogous to how Wood’s construction assigns an edge of a graph to each $\xi_i^{2^j}$. Can this analogy be leveraged to obtain further graph theoretic results  in both the $p$-primary and $C_2$-equivariant settings\color{black}? 
    \item Can we extend our graphical constructions to represent comodules over the quotients $A^\vee(n)$ or $A^\vee_{C_2}(n)$\color{black}?
\end{enumerate}

\subsection{Outline of the paper} In \cref{sec:Woodp=2} we recall Wood's construction of the graphs corresponding to monomials $x \in A^\vee (n)$ when $p=2$ \cite[\textsection 8]{Wood98}. We also discuss graph theoretic results in this setting due to Larson and Yearwood \cite{Larson22,Yearwood19}. In \cref{sec:oddprimes} we describe a construction yielding graphs corresponding to monomials $x \in A^\vee (n)$ when $p$ is an odd prime. This section also contains the proofs of our odd primary results. In \cref{sec:C2} we extend our odd primary construction and results to the $C_2$-equivariant setting. 

\subsection{Definitions, conventions, and notation} \label{sec:notation}
We recall some foundational definitions and set notation.

\begin{enumerate}
    \item A \textbf{\textit{graph}} $G=(V_G,E_G)$ is a structure that contains a set of objects $V_G$, which are called \textbf{\textit{vertices}}, and relations on the vertices coming from $E_G$, which are called \textbf{\textit{edges}}. The subscript $G$ on $V_G$ and $E_G$ will often be suppressed when the graph is clear from context.

    \item Given an edge $e \in E_G$ relating two vertices $v_0$ and $v_1,$ we say $v_0$ and $v_1$ are the \textbf{\textit{ends}} of $e$. 

    \item For $e\in E$ connecting vertices $u,v\in V$, we will write $e\colon u\to v$ and for $n$ copies of $e$, we will write $e^n\colon u\xrightarrow{n} v$. 

    \item The \textbf{\textit{degree}} of a vertex $v$ is given by the number of edges that touch that vertex, we will denote this by $\deg(v)$. 

    \item A \textbf{\textit{walk}} is a sequence of edges and vertices in a graph. When the starting and ending vertices are the same, we say the walk is a \textbf{\textit{closed walk}}.

    \item  A \textbf{\textit{path}} is a sequence of edges and vertices where no vertex is visited more than once, except possibly the start and end vertices. 

    \item Note that a path is a specific type of walk, so when the ending and starting vertices are the same, the path is considered a \textbf{\textit{closed path}}.

    \item A \textbf{\textit{cycle}}, also known as a closed path, is a path in a graph that begins and ends at the same vertex. 

    \item  A \textbf{\textit{connected}} graph is one for which there is a path between any given pair of vertices. 

    \item A \textbf{\textit{tree}} is a graph that is connected and acyclic. 

    \item A \textbf{\textit{complete}} graph is a graph where any given pair of vertices is connected by an edge. 

    \item A \textbf{\textit{Hamiltonian path}} is a path that visits every vertex exactly once. 

    \item A \textbf{\textit{Hamiltonian cycle}} is a cycle that visits every vertex exactly once. 

    \item A \textbf{\textit{directed edge}} is an edge where the ends are distinguished - one is the head and one is the tail. In particular, a directed edge is specified as an ordered pair of vertices (tail, head).

    \item A \textbf{\textit{directed graph}} is a graph where every edge is directed. 

    \item A directed graph $D$ is \textbf{\textit{unilateral}} if for every pair of distinct vertices $v_i, \, v_j \in D$, there is a directed path starting at $v_i$ and ending at $v_j$ or vice versa. 

    \item Given a graph G with an ordered vertex set $V_G = \{v_1, v_2, \cdots, v_n \}$, the associated \textbf{\textit{adjacency matrix}} $A_G$ is an $n \times n$ matrix with entry $a_{i,j}$ given by the number of edges from $v_i$ to $v_j$.
\end{enumerate}

\subsection*{Acknowledgments}
The authors are grateful to the Department of Mathematics at the University of Colorado Boulder for sponsoring the Research Experience for Undergraduates and Graduates in Summer 2025, which led to this project. The fourth author would also like to thank the Isaac Newton Institute for Mathematical Sciences, Cambridge, for support and hospitality during the programme Equivariant homotopy theory in context, where partial work on this paper was undertaken. This work was supported by EPSRC grant EP/Z000580/1. The material in this paper is also based upon work supported by the National Science Foundation under Grant No. DMS 2135884.

\section{Wood's graphical construction at the prime $2$} \label{sec:Woodp=2}

In this section, we recall Wood's graphical construction at the prime $2$ and describe results due to Yearwood \cite{Yearwood19} and Larson \cite{Larson22}. These results inspire our work at odd primes and in the $C_2$-equivariant setting so it is helpful recollect them here.

\subsection{The mod $2$ dual Steenrod algebra}

The mod $2$ dual Steenrod algebra is 
\[
A^\vee \cong \mF_2 [ \xi_1, \xi_2, \cdots],
\]
where $| \xi_n | = 2^n - 1$. The coproduct is given by 
\[
\psi(\xi_n) = \sum_{i=0}^n \xi_{n-i}^{2^i} \otimes \xi_i
\]
and the antipode $c: A^\vee \to A^\vee$ is given recursively by 
\begin{align*}
    c(\xi_0) = 1 \\
    \sum_{k=0}^i \xi_{i-k}^{2^k} c(\xi_k) = 0.
\end{align*}
 Milnor solved this recursion in \cite[Lemma 10]{Milnor58} obtaining
\[
c(\xi_i) = \sum_\pi \prod_{k=1}^{\ell(\pi)} \xi_{\pi(k)}^{2^{\sigma(k)}}
\]
where the sum is over all ordered partitions $\pi$ of $i$, $\ell(\pi)$ is the length of $\pi$, $\pi(k)$ is the $k$th term of $\pi$, and $\sigma(k)$ is the sum of the first $k-1$ terms of $\pi$.

The Hopf algebra quotient $A^\vee(n)$ is given by 
\[
A^\vee (n) \cong \bigslant{A^\vee}{(\xi_1^{2^{n + 1}}, \xi_2^{2^n}, \xi_3^{2^{n-1}}, \cdots, \xi_{n + 1}^2, \xi_{n + 2}, \xi_{n + 3}, \cdots)}.
\]
In other words, 
\[
A^\vee(n) \cong \bigslant{\mF_2 [ \xi_1, \xi_2, \cdots, \xi_{n + 1}]}{(\xi_1^{2^{n + 1}}, \xi_2^{2^n}, \xi_3^{2^{n-1}}, \cdots, \xi_{n + 1}^2)}.
\]

\subsection{Wood's graphical construction $(p = 2)$} 
Suppose $x$ is a monomial in $A^\vee(n)$. Then 
\[
x = \xi_1^{r_1} \xi_2^{r_2} \cdots \xi_n^{r_n} \xi_{n + 1}^{r_{n + 1}},
\]
where $r_i \in \{0,1, \cdots, 2^{n + 1 - i}-1 \}$ and where each $r_i$ has the $2$-adic expansion
\[
r_i = \sum_{m=0}^{n + 1-i} c_{i,n+ 1-i-m} \cdot 2^{n+ 1-i-m}.
\]
We can rewrite the monomial $x$ as 
\[
x = \prod_{i=1}^{n + 1} \prod_{m = 0}^{{n+1-i}} \xi_i^{c_{i, n + 1 - i - m} \cdot 2^{n + 1 - i - m}}.
\]
This rewritten form of $x$ is useful for constructing the graphical interpretation of monomials in $A^\vee (n)$ so we make the following definition.
\begin{deff}\label{sform1}
    Given  a monomial \color{black} $x \in A^\vee (n)$, the \textbf{standard form} of $x$ is 
    \[
x = \prod_{i=1}^{n + 1} \prod_{m = 0}^{{n+1-i}}\xi_i^{c_{i, n + 1 - i - m} \cdot 2^{n + 1 - i - m}},
\]
where the exponents on the $\xi_i$ are written in terms of their $2$-adic expansion. 
\end{deff}

Given a monomial \color{black} $x \in A^\vee(n),$ Wood defines the graph $G_x = (V_x, E_x)$ by giving the ordered vertex set $V_x = \{2^0, 2^1, 2^2, \cdots, 2^{n + 1} \}$ and the edge set 
\[
E_x = \bigcup \left\{ \xi_i^{c_{i,j} \cdot 2^j} : 2^{j } \xrightarrow{c_{i,j}} 2^{i + j}  \right\}
\]
\cite[\textsection 8]{Wood98}. 
Notice that  the degree $|\xi_i^{2^j}|$ of the edge $\xi_i^{2^j}$, that is, the grading of $\xi_i^{2^j}$ when viewed as an element of $A^\vee(n)$, is \color{black}the absolute value of the difference of the ends of $\xi_i^{2^j}$.

Considering entries in the upper triangle of the adjacency matrix $A_{x}$ associated to the graph $G_x$, we observe that 
\begin{equation} \label{eq:acRel}
    a_{l,k} = c_{k-l,l-1} \qquad (\text{when } l < k),
\end{equation}
that is,
\begin{equation} \label{eq:AdjMp=2}
A_{x} = \begin{bmatrix}
    0 & c_{1,0} & c_{2,0} & \quad & c_{n,0} & c_{n + 1, 0} \\
    c_{1,0} & 0 & c_{1,1} & \quad & c_{n-1,1} & c_{n, 1} \\
    c_{2,0} &  c_{1,1} & 0 & \quad & c_{n-2,2} & c_{n -1, 2} \\
    & & & \ddots & & \\
    c_{n,0} & c_{n-1,1} & c_{n-2,2} & & 0 & c_{1, n} \\
    c_{n + 1, 0} & c_{n, 1} & c_{n-1,2} & & c_{1, n} & 0 \\
\end{bmatrix}.
\end{equation}

\begin{example}
The monomial $\xi_1^2 \xi_2^2 \in A^\vee(2)$ has associated graph $G_{\xi_1^2 \xi_2^2}$:
    
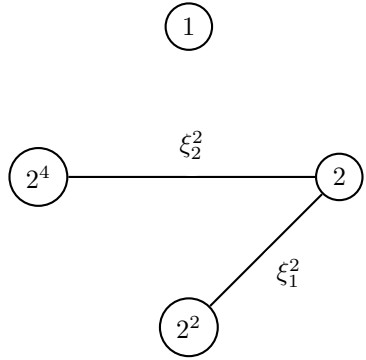
\vspace{1em}
        \begin{center}
        \begin{tikzpicture}
            \begin{scope}[every node/.style={circle,thick,draw}]
            \node (1) at (0,2) {$1$};
            \node (20) at (2,0) {$2$};
            \node (21) at (0,-2) {$2^2$};
            \node (22) at (-2,0) {$2^3$};
        \end{scope}
        \begin{scope}[>={Stealth[black]},
                      every node/.style={fill=white,circle},
                      every edge/.style={draw, thick}]
            \path[] [below right] (20) edge node {\color{black} $\xi_1^{2}$} (21);
            \path [above] (20) edge node {$\xi_2^{2}$} (22);
        \end{scope}
    \end{tikzpicture}
    \end{center}
\end{example}

\begin{example}
    The monomial $\xi_1^7 \xi_2 \xi_3 \in A^\vee (2)$ has associated graph $G_{\xi_1^7 \xi_2 \xi_3}${:}
    
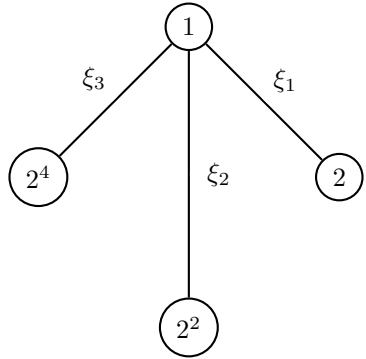
\vspace{1em}
        \begin{center}
        \begin{tikzpicture}
            \begin{scope}[every node/.style={circle,thick,draw}]
            \node (1) at (0,2) {$1$};
            \node (20) at (2,0) {$2$};
            \node (21) at (0,-2) {$2^2$};
            \node (22) at (-2,0) {$2^3$};
        \end{scope}
        \begin{scope}[>={Stealth[black]},
                      every node/.style={fill=white,circle},
                      every edge/.style={draw, thick}]
            \path[] [above right] (1) edge node {\color{black} $\xi_1$} (20);
            \path (1) edge node {$\xi_2$} (21);
            \path [above left] (1) edge node {$\xi_3$} (22);
            \path[below right] (20) edge node {$\xi_1^2$} (21);
            \path[below left] (21) edge node {$\xi_1^4$} (22);
        \end{scope}
    \end{tikzpicture}
    \end{center}
\end{example}

\subsection{Connectedness criteria} In \cite{Larson22}, Larson gives connectedness criteria for the graph $G_x$ given the only the data of the monomial $x$. We recall some of Larson's results here for comparison with our odd primary and $C_2$-equivariant results. 

Throughout this section we will assume $x$ is a given monomial in $A^\vee(n)$. Thus we can write 
\[
x = \xi_1^{r_1} \xi_2^{r_2} \cdots \xi_n^{r_n} \xi_{n + 1}^{r_{n + 1}}
\]
for some $r_i \in \{0,1,\cdots, 2^{n + 1 - i} - 1\}$. We will also make use of $x$ written { in standard form}, that is writing
\[
x = \prod_{i=1}^{n + 1} \prod_{m = 0}^{{n+1-i}} \xi_i^{c_{i, n + 1 - i - m} \cdot 2^{n + 1 - i - m}}.
\]

Recall that a graph is \textit{connected} if for any two distinct vertices there exists a path between them. Given two vertices of a graph, we can count the number of paths between of a given length. To count the number of acyclic paths between two vertices of a given length, say of $t$ edges, we define the following set. 

\begin{deff} \label{def:Tt}
    Suppose $t$ is an integer with $1 \leq t \leq n + 1$. Define $T_t$ to be the set of all $(t+1)$-tuples $(b_0, b_1, ,\cdots b_{t - 1}, b_t) $ where $b_i \in \{0, 1, \cdots, n\}$ {and} no two entries are the same.
\end{deff}

We can interpret tuples $(b_0, b_1, \cdots b_{t - 1}, b_t) \in T_t$ as corresponding to acyclic paths of length $t$ from the vertex $2^{b_0}$ to the vertex $2^{b_t}$. Specifically, the tuple $(b_0, b_1, \cdots b_{t - 1}, b_t)$ represents a path through the vertices $2^{b_0}, 2^{b_1},  \cdots , 2^{b_{t-1}}, 2^{b_t}$. Note that $b_{k - 1}$ is not necessarily less than $b_k$ so the path need not pass through the vertices of $G_x$ in ascending order.

Thus, if we are interested in indicating whether a particular acyclic path exists in our graph, we can begin by checking whether there is an edge between the vertices $2^{b_{k-1}}$ and $2^{b_k}$. This is indicated by the value of 

\[
a_{\max (b_{k-1}, b_k), \min (b_{k-1}, b_k)} = \begin{cases}
    1 & \text{ if there is an edge between } 2^{b_{k - 1}} \text{ and } 2^{b_k} \\
    0 & \text{ otherwise }
\end{cases},
\]
in the adjacency matrix $A_x$. Using the observation that $a_{l,k} = c_{k - l, l - 1}$ when $l < k$ (\ref{eq:acRel}), we can rewrite $a_{\max (b_{k-1}, b_k), \min (b_{k-1}, b_k)}$ in terms of the coefficients $c_{i,j}$ appearing in the standard form of $x$. Specifically,
\[
a_{\max (b_{k-1}, b_k), \min (b_{k-1}, b_k)} = c_{|b_k-b_{k-1}|, \max (b_{k-1}, b_k)-1}
\]
Thus the path corresponding to $(b_0, b_1, \cdots, b_{t - 1}, b_t)$ exists in our graph if and only if
\[
\prod_{k = 1}^{t} c_{|b_k-b_{k-1}|, \max (b_{k-1}, b_k)-1}  = 1. 
\]
Therefore, to count the total number of acyclic paths between two given vertices, say the vertex $2^{r}$ and the vertex $2^{s}$ in our graph $G_x$, we  make the following definitions.
\begin{deff} \label{def:T't}
    Given integers $r, s \in \{0,1, \cdots, n\}$, define $T_t'$ to be the set of all $(t + 1)$-tuples $(r, b_1, \cdots, b_{t -1 }, s)$ where $b_i \in \{ 0, 1, \cdots, n + 1\}$. 
\end{deff}
In particular, $T_t'$ is the subset of $T_t$ where the first element in each $(t + 1)$-tuple is set equal to $r$ and the last element in each $(t + 1)$-tuple is set equal to $s$.
\begin{deff} 
   Given integers $r, s \in \{0,1, \cdots, n\}$, let 
    \[
    { C(r, s)} : = \sum_{t = 1}^{n + 1} \sum_{T_t'} \prod_{k = 1}^t c_{|b_k-b_{k-1}|, \max (b_{k-1}, b_k)-1}.
    \]
\end{deff}
\color{black}
Here, the outer sum is taken over all possible lengths of paths between the vertices. Since there are $n + 2$ vertices, the longest path that is not a cycle has length $n + 1$. The inner sum is taken over $T_t'$, that is, over all the acyclic paths of length $t$ between the vertices $2^{r}$ and $2^{s}$.  

Using this notation, Larson proves the following connectedness criterion.
\begin{theorem}[{\cite[Theorem 1.1]{Larson22}}]
    The graph $G_x$ is connected if and only if the integers
    \[
    C(r,s) : = \sum_{t = 1}^{n + 1} \sum_{T_t'} \prod_{k = 1}^t c_{|b_k-b_{k-1}|, \max (b_{k-1}, b_k)-1}
    \]
    are positive for all integers $r$ and $s$ with $0 \leq r < s \leq n + 1$.
\end{theorem}

\subsection{Trees and Hamilton cycles}

Letting $\alpha (m)$ denote the number of nonzero digits in the base $2$ expansion of an integer $m,$ Larson also proves:
\begin{theorem}[{\cite[Theorem 1.2]{Larson22}}]
    The graph $G_x$ is a tree if and only if $G_x$ is connected and 
    \[
    \sum_{i=1}^{n + 1} \alpha(r_i) = n + 1.
    \]
\end{theorem}

In \cite{Yearwood19}, Yearwood observed that the graphs $G_x$ can naturally be viewed as directed graphs with edges oriented in the direction of the larger vertex (i.e., $2^\alpha \to 2^\beta$ if $0 \leq \alpha < \beta \leq n + 1$). We will denote the directed version of the graph $G_x$ by $G_x^{\textup{dir}}$. With the directed version of Wood's construction in mind, Larson studied the unilaterality, the appropriate connectedness property for directed graphs (see \cref{sec:notation}). { To count the total number of acyclic directed paths between two given vertices, say the vertex $2^r$ and the vertex $2^s$, we make the following definitions. 

\begin{deff}\label{def:T''t}
     Given integers $r, s \in \{0,1, \cdots, n\}$, let $T_t''$ be the set of all $(t + 1)$-tuples $(b_0, b_1, \cdots, b_t)$ where $b_i \in \{ 0, 1, \cdots, n + 1\}$ and $r = b_0 < b_1 < \cdots < b_t = s$.
\end{deff}
In particular, $T_t''$ is the subset of $T_t'$ where the elements in each $(t + 1)$-tuple are strictly increasing. 
\begin{deff} 
    Given integers $r, s \in \{0,1, \cdots, n\}$, let 
    \[
    {U(r,s)}: = \sum_{t = 1}^{n + 1} \sum_{{T}_t''} \prod_{k = 1}^t c_{b_{k - 1} , b_k}.
    \]
\end{deff}
Then $U(r,s)$ is the directed graph analogue of $C(r,s)$ and thus we can think of the outer and inner sums similarly. 

Using this notation, Larson proves the following characterization of unilaterality.}

\begin{theorem}[{\cite[Theorem 1.3]{Larson22}}]
    The directed graph $G_x^\textup{dir}$ is unilateral if and only if ${U(r,s)}$
    is positive for all integers $r,s$ with $0 \leq r < s \leq n + 1$.
\end{theorem}

Larson also studied gave an algebraic criteron describing when a Hamilton cycle must occur in the graph $G_x$. 
\begin{theorem}[{\cite[Theorem 1.4]{Larson22}}]
    The graph $G_x$ has a Hamilton cycle if $n > 0$ and for every vertex $2^i$ of $G_x,$ 
    \[
    \# \{ 1 \leq k \leq j : c_{j, j - k} = 1  \} + \# \{1 \leq k \leq n + 1 - j : a_{j + k, j} = 1 \} \geq \frac{n}{2}.
    \]
    Moreover, the directed graph $G_x^\textup{dir}$ has a directed Hamilton path if and only if $x$ is divisible by $\xi_1^{2^{n + 1} - 1}$.
\end{theorem}

\subsection{The Hopf algebra structure }
Following Lemmas 3.1.7 and 3.1.8 of \cite{Yearwood19}, Larson also gave a graph theoretic interpretation of the coproduct and antipode on monomials $\xi_i^{2^j} \in A^\vee(n)$. 
\begin{theorem}[{\cite[Theorem 1.5]{Larson22}}] \label{thm:LarsonCoproduct}
    The coproduct $\psi(\xi_i^{2^j}) \in A^\vee(n)$ is the sum of tensors of all pairs of edges that make length $2$ directed paths from the vertex $2^j$ to the vertex $2^{i +j}$ 
     in the complete graph on the vertices $\{2^0, 2^1, 2^2, \cdots, 2^{n + 1}\}$ considered as a directed graph\color{black}. Moreover, the antipode $c(\xi_i^{2^j}) \in A^\vee(n)$ is the sum of all directed paths from the vertex $2^j$ to the vertex $2^{i + j}$. 
\end{theorem}
And as a corollary, Larson obtained another characterization of the unilaterality of $G_x^\textup{dir}$. 
\begin{cor}[{\cite[Corollary 1.6]{Larson22}}]
    For $x \in A^\vee(n)$, the graph $G_x^\textup{dir}$ is unilateral if and only if for each $\xi_i^{2^j} \in A^\vee (n)$ at least one summand of $c(\xi_i^{2^j})$ is a factor of $x$. 
\end{cor}

\section{A graphical construction at odd primes} \label{sec:oddprimes}

We now define a graphical construction for the mod $p$ dual Steenrod algebra quotients $A^\vee(n)$, where $p$ is an odd prime, with the goal of extending Wood's construction at the prime $p = 2$. Recall the mod $p$ dual Steenrod algebra
$$A^\vee
=\mF_p[\xi_1,\xi_2,\dots]\otimes E[\tau_1,\tau_2,\dots],$$ where $|\xi_i|=2(p^i-1)$ and $|\tau_i|=2p^i-1$. Since we work one prime at a time, we suppress $p$ from our notation simply writing $A^\vee$ for the mod $p$ dual Steenrod algebra. 

For each integer $n\ge1$, we study graphs arising from monomials in the quotients
$$A^\vee(n) = \bigslant{\mathbb{F}_p[\xi_1, \xi_2, \cdots, \xi_n, \tau_o, \tau_1, \cdots, \tau_n]}{(\xi_1^{p^n}, \xi_2^{p^{n-1}}, \cdots, \xi_n^{p}, \tau_0^2, \tau_1^2, \cdots, \tau_n^2)}.$$ 

Since addition is represented by disjoint union of graphs, it is sufficient to restrict our study to monic monomials $x \in A^\vee (n)$. Given a monic monomial $x \in A^\vee(n)$, we can write 
$$x = \tau_0^{\epsilon_0} \tau_1^{\epsilon_1} \cdots \tau_n^{\epsilon_n} \xi_1^{r_1} \xi_2^{r_2} \cdots \xi_{n}^{r_{n}}$$ 
where $\epsilon_i\in\{0,1\}$ and $0\le r_i\le p^{n+1-i}-1$. Following Wood's construction, we will utilize the $p$-adic expansion of $r_i$, 
$$r_i=\sum_{m=0}^{n-i}c_{i,n-i-m}\cdot p^{n-i-m}.$$
Extending \cref{sform1}, we make the following definition. 
\begin{deff}\label{def:StandardFormOdd}
    For a monic monomial $x\in A^\vee(n)$, the \textit{\textbf{standard form}} of $x$ is 
    $$\tau_0^{\epsilon_0}\tau_1^{\epsilon_1}\cdots\tau_n^{\epsilon_n}\prod_{i=1}^n\prod_{m=0}^{n-i}\xi_i^{c_{i,n-i-m}\cdot p^{n-i-m}},
    $$ where the exponents on the $\xi_i$ are expanded into their $p$-adic expansion.
\end{deff} 

Given a monic monomial $x\in A^\vee(n)$, we define the graph $G_x = (V_x, E_x)$ setting the ordered vertex set $V_x=\left\{1, 2\cdot p^0, 2\cdot p^1, \dots, 2\cdot p^n\right\}$ and the edge set
\[
E_x=\left\{\tau_i\colon 1\to 2p^i \right\}_{0\le i\le n} {\cup}\left\{\xi_i^{c_{i,j}\cdot p^j}\colon 2p^j\xrightarrow{c_{i,j}} 2p^{i + j} \right\}_{\substack{1\le i\le n \\0\le j\le n-i}}
\]
where each $\tau_i$ denotes an edge connecting vertices $1$ and $2p^i$ and each $\xi_i^{c_{i,j}\cdot p^j}$ denotes $c_{i,j}$ edges connecting vertices $2p^j$ and $2p^{i + j}$. Notice that, similarly to the $p = 2$ case,  the degree $|\xi_i^{2^j} |$ of the edge $\xi_i^{2^j}$, that is, the grading of $\xi_i^{2^j}$ when viewed as an element of $A^\vee(n)$, is the absolute value of the difference of its ends. Additionally, the degree $| \tau_i |$ of the edge $\tau_i$ is the absolute value of the difference of its ends. \color{black}

\begin{example}
    The monomial $x=\xi_1^{2p}\xi_2 \in A^\vee (2)$ has associated graph $G_x$:
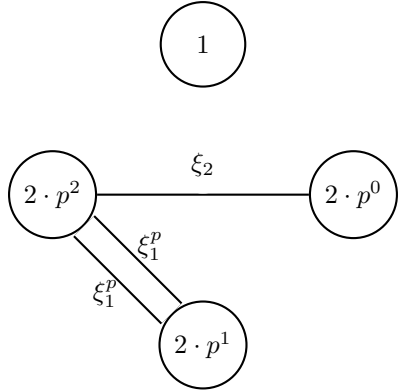
    \begin{center}
    \begin{tikzpicture}
        \begin{scope}[every node/.style={circle,thick,draw}]
            \node (1) at (0,2) {$\phantom{m}1\phantom{m}$};
            \node (20) at (2,0) {$2\cdot p^0$};
            \node (21) at (0,-2) {$2\cdot p^1$};
            \node (22) at (-2,0) {$2\cdot p^2$};
        \end{scope}
        \begin{scope}[>={Stealth[black]},
                      every node/.style={fill=white,circle},
                      every edge/.style={draw, thick}]
            \path [above] (20) edge node {$\xi_2$} (22);
        \end{scope}
        \draw[-, double distance=10pt, thick] (21) --
             node[above right]{$\xi_1^{p}$}
             node[below left]{$\xi_1^{p}$} (22);
    \end{tikzpicture}
    \end{center}
    Notice that in the graph $G_x$, there are $2 = c_{1,1}$ edges connecting vertices $2p^1$ and $2p^2$. These correspond to the two factors of $\xi_1^p$. In general, our graphs may have up to $p-1$ factors of $\xi_1^p$, in which case there would be $p-1$ edges connecting the vertices $2p^1$ and $2p^2$.  
\end{example}    

\begin{example}
    The monomial $x=\tau_0\tau_1\tau_2 \in A^\vee(2)$ has associated graph $G_{\tau_0\tau_1\tau_2}$:
    
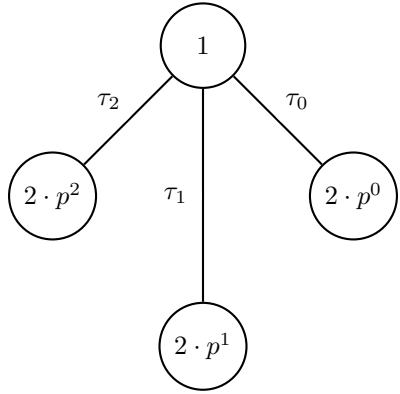
\begin{center}
        \begin{tikzpicture}
        \begin{scope}[every node/.style={circle,thick,draw}]
            \node (1) at (0,2) {$\phantom{m}1\phantom{m}$};
            \node (20) at (2,0) {$2\cdot p^0$};
            \node (21) at (0,-2) {$2\cdot p^1$};
            \node (22) at (-2,0) {$2\cdot p^2$};
        \end{scope}
        \begin{scope}[>={Stealth[black]},
                      every node/.style={fill=white,circle},
                      every edge/.style={draw, thick}]
            \path [above right] (1) edge node {$\tau_0$} (20);
            \path [left] (1) edge node {$\tau_1$} (21);
            \path [above left] (1) edge node {$\tau_2$} (22);
        \end{scope}
    \end{tikzpicture}
    \end{center}
\end{example}

Consider a monic monomial $x \in A^\vee(n)$. We observe that the entries in the upper triangle of the adjacency matrix $A_{x}$, that is $a_{l,k}$ where $l<k$, are given by
\begin{equation} \label{eq:acOddRel}
a_{l,k}=
\begin{cases}
    \epsilon_{k-2} & l=1 \\
    c_{k-l,l-2} & 1<l<k
\end{cases}.
\end{equation}
It may be helpful to recall that the values of $\epsilon_{k - 2}$ and $c_{k - l, l - 2}$ come from the exponents of $x$ in standard form (\cref{def:StandardFormOdd}). Therefore,
\begin{equation} \label{eq:AdjMatrixOdd}
    A_{x}=\begin{bmatrix}
    0 & \epsilon_0 & \epsilon_1 & \epsilon_2 & & \epsilon_{n-1} & \epsilon_n \\
    \epsilon_0 & 0 & c_{1,0} & c_{2,0} & & c_{n-1,0} & c_{n,0} \\
    \epsilon_1 & c_{1,0} & 0 & c_{1,1} & & c_{n-2,1} & c_{n-1,1} \\
    \epsilon_2 & c_{2,0} & c_{1,1} & 0 & & c_{n-3,2} & c_{n-2,2} \\
     &  &  &  & \ddots & \\
    \epsilon_{n-1} & c_{n-1,0} & c_{n-2,1} & c_{n-3,2} & & 0 & c_{1,n-1} \\
    \epsilon_n & c_{n,0} & c_{n-1,1} & c_{n-2,2} & & c_{1,n-1} & 0 \\
\end{bmatrix}
\end{equation}
Notice that the $i^{\text{th}}$ superdiagonal, excluding the first row, contains the coefficients of the $p$-adic expansion for $\xi_i$. Reading $i^{\text{th}}$ superdiagonal from bottom to top yields the base $p$ expansion of $r_i$:
$$r_i={\left(c_{i,n-i} \, c_{i,n-i-1} \, \dots \, c_{i,1} \, c_{i,0} \right)_p}.$$

\subsection{Examples} In this sections, all examples have prime $p$ set equal to $3$.

\begin{example}[A maximally complete graph] 
    Consider the monomial 
    \[
    x = \tau_0 \tau_1 \tau_2 \tau_3 \xi_1^{26} \xi_2^8 \xi_3^2 \in A^\vee (3).
    \]
    Rewriting $x$ in standard form yields $$x =\tau_0\tau_1\tau_2\tau_3\xi_1^{ 2 \color{black} \cdot 3^2}\xi_1^{ 2 \color{black} \cdot 3^1}\xi_1^{ 2 \color{black} \cdot 3^0}\xi_2^{ 2 \color{black} \cdot3^1}\xi_2^{ 2 \color{black} \cdot3^0}\xi_3^{ 2 \color{black} \cdot3^0} \in A^\vee(3)$$
    with associated graph and adjacency matrix depicted below.
    \vspace{1em}
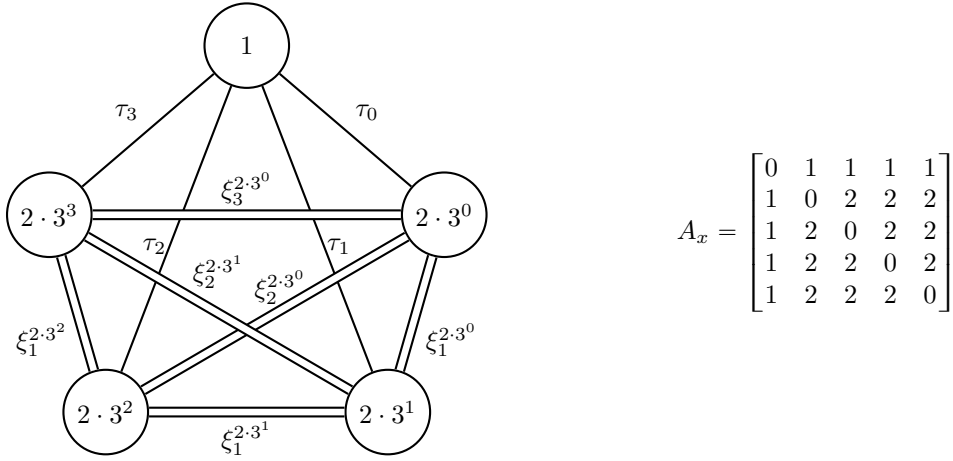
    \begin{center}
        \begin{minipage}[c]{.55\textwidth}
            \begin{tikzpicture}[scale=.75]
                \begin{scope}[every node/.style={circle,thick,draw}]
                    \node (1) at (0,3) {$\phantom{m}1\phantom{m}$};
                    \node (20) at (3.5,0) {$2\cdot 3^0$};
                    \node (21) at (2.5,-3.5) {$2\cdot 3^1$};
                    \node (22) at (-2.5,-3.5) {$2\cdot 3^2$};
                    \node (23) at (-3.5,0) {$2\cdot 3^3$};
                \end{scope}
                \begin{scope}[>={Stealth[black]},
                              every node/.style={},
                              every edge/.style={draw, thick}]
                    \path [above right] (1) edge node {$\tau_0$} (20);
                    \path [below right] (1) edge node {$\tau_1$} (21);
                    \path [below left] (1) edge node {$\tau_2$} (22);
                    \path [above left] (1) edge node {$\tau_3$} (23);
                \end{scope}
                    \draw[-, double distance = 2.5pt, thick] (20) --
                        node[above left]{}
                        node[below right]{$\xi_1^{{2}\cdot3^0}$} (21);
                    \draw[-, double distance = 2.5pt, thick] (21) --
                        node[above]{}
                        node[below]{$\xi_1^{{2}\cdot3^1}$} (22);
                    \draw[-, double distance = 2.5pt, thick] (22) --
                        node[below left]{$\xi_1^{{2}\cdot3^2}$}
                        node[above right]{} (23);
                    \draw[-, double distance = 2.5pt, thick] (20) --
                        node[above,xshift=2pt]{$\xi_2^{{2}\cdot3^0}$}
                        node[below right]{} (22);
                    \draw[-, double distance = 2.5pt, thick] (21) --
                        node[below]{}
                        node[above,yshift=6pt]{$\xi_2^{{2}\cdot3^1}$} (23);
                    \draw[-, double distance = 2.5 pt, thick] (20) --
                        node[above]{$\xi_3^{{2}\cdot3^0}$}
                        node[below]{} (23);
            \end{tikzpicture}
        \end{minipage}
        \hspace{1cm}
        \begin{minipage}[c]{.3\textwidth}
            $A_{x}=\begin{bmatrix}
                0 & 1 & 1 & 1 & 1 \\
                1 & 0 & { 2} & {2} & { 2} \\
                1 & {2} & 0 & { 2} & { 2} \\
                1 & { 2} & { 2} & 0 & { 2} \\
                1 & { 2} & { 2} & { 2} & 0 \\
            \end{bmatrix}$
        \end{minipage}
    \end{center}
\end{example}

\begin{example}[A minimally complete graph] 
    Consider the monomial 
    \[
    x = \tau_0 \tau_1 \tau_2 \tau_3 \xi_1^{13} \xi_2^4 \xi_3 \in A^\vee(3).
    \]
    Rewriting $x$ in standard form yields
   $$\tau_0 \tau_1 \tau_2 \tau_3 \xi_1^{3^2} \xi_1^{3^1} \xi_1^{3^0} \xi_2^{3^1} \xi_2^{3^0} \xi_3^{3^0} \in A^\vee(3)$$
   with associated graph and adjacency matrix depicted below. 
   \begin{center}
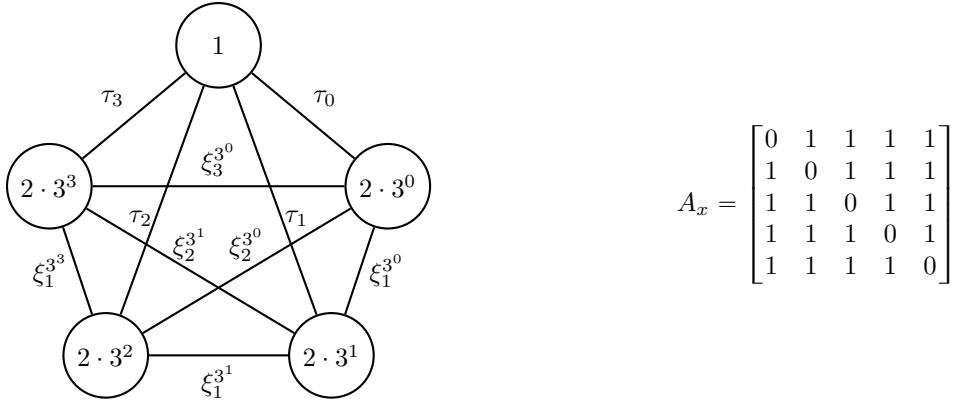
        \begin{minipage}[c]{.55\textwidth}
            \begin{tikzpicture}[scale=.75]
                \begin{scope}[every node/.style={circle,thick,draw}]
                    \node (1) at (0,2.5) {$\phantom{m}1\phantom{m}$};
                    \node (20) at (3,0) {$2\cdot 3^0$};
                    \node (21) at (2,-3) {$2\cdot 3^1$};
                    \node (22) at (-2,-3) {$2\cdot 3^2$};
                    \node (23) at (-3,0) {$2\cdot 3^3$};
                \end{scope}
                \begin{scope}[>={Stealth[black]},
                              every node/.style={},
                              every edge/.style={draw, thick}]
                    \path [above right] (1) edge node {$\tau_0$} (20);
                    \path [below right] (1) edge node {$\tau_1$} (21);
                    \path [below left] (1) edge node {$\tau_2$} (22);
                    \path [above left] (1) edge node {$\tau_3$} (23);
                    \path [right] (20) edge node {\color{black} $\xi_1^{3^0}$} (21);
                    \path [below] (21) edge node {\color{black} $\xi_1^{3^1}$} (22);
                    \path [left] (22) edge node {\color{black} $\xi_1^{3^3}$} (23);
                    \path [above] (20) edge node {\color{black} $\xi_2^{3^0}$} (22);
                    \path [above] (21) edge node {\color{black} $\xi_2^{3^1}$} (23);
                    \path [above] (20) edge node {\color{black} $\xi_3^{3^0}$} (23);
                \end{scope}
            \end{tikzpicture}
        \end{minipage}
        \hspace{1cm}
        \begin{minipage}[c]{.3\textwidth}
            $A_{x}=\begin{bmatrix}
                0 & 1 & 1 & 1 & 1 \\
                1 & 0 & 1 & 1 & 1 \\
                1 & 1 & 0 & 1 & 1 \\
                1 & 1 & 1 & 0 & 1 \\
                1 & 1 & 1 & 1 & 0 \\
            \end{bmatrix}$
        \end{minipage}
    \end{center}
\end{example}

\begin{example}[A connected graph] 
    Consider the monomial $x = \tau_0 \tau_3 \xi_1^{4} \xi_2^6 \in A^\vee(3)$. Rewriting $x$ in standard form yields
   $$\tau_0 \tau_3 \xi_1^{3^1} \xi_1^{3^0} \xi_2^{2\cdot 3^1} \in A^\vee(3)$$
   with associated graph and adjacency matrix depicted below. 
   \begin{center}
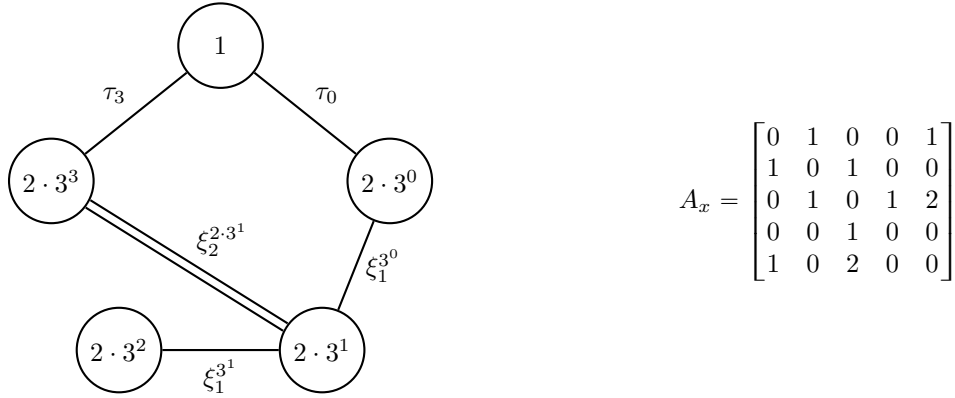
        \begin{minipage}[c]{.55\textwidth}
            \begin{tikzpicture}[scale=.9]
                \begin{scope}[every node/.style={circle,thick,draw}]
                    \node (1) at (0,2) {$\phantom{m}1\phantom{m}$};
                    \node (20) at (2.5,0) {$2\cdot 3^0$};
                    \node (21) at (1.5,-2.5) {$2\cdot 3^1$};
                    \node (22) at (-1.5,-2.5) {$2\cdot 3^2$};
                    \node (23) at (-2.5,0) {$2\cdot 3^3$};
                \end{scope}
                \begin{scope}[>={Stealth[black]},
                              every node/.style={},
                              every edge/.style={draw, thick}]
                    \path [above right] (1) edge node {$\tau_0$} (20);
                    \path [above left] (1) edge node {$\tau_3$} (23);
                    \path [below] (21) edge node {\color{black} $\xi_1^{3^1}$} (22);
                    \path [right] (20) edge node {\color{black} $\xi_1^{3^0}$} (21);
                \end{scope}
                     \draw[-, double distance = 2.5pt, thick] (21) --
                        node[above right]{$\xi_2^{{2}\cdot 3^1}$}
                        node[below left]{} (23);
            \end{tikzpicture}
        \end{minipage}
        \hspace{1cm}
        \begin{minipage}[c]{.3\textwidth}
            $A_{x}=\begin{bmatrix}
                0 & 1 & 0 & 0 & 1 \\
                1 & 0 & 1 & 0 & 0 \\
                0 & 1 & 0 & 1 & { 2} \\
                0 & 0 & 1 & 0 & 0 \\
                1 & 0 & { 2} & 0 & 0 \\
            \end{bmatrix}$
        \end{minipage}
    \end{center}
\end{example}

\begin{example}[A tree graph]
    Consider the monomial $x = \tau_0 \tau_3 \xi_2^4 \in A^\vee(3)$. Rewriting $x$ in standard form yields
   $$\tau_0 \tau_3 \xi_2^{3^1} \xi_2^{3^0} \in A^\vee(3)$$
   with associated graph and adjacency matrix depicted below. 
   \begin{center}
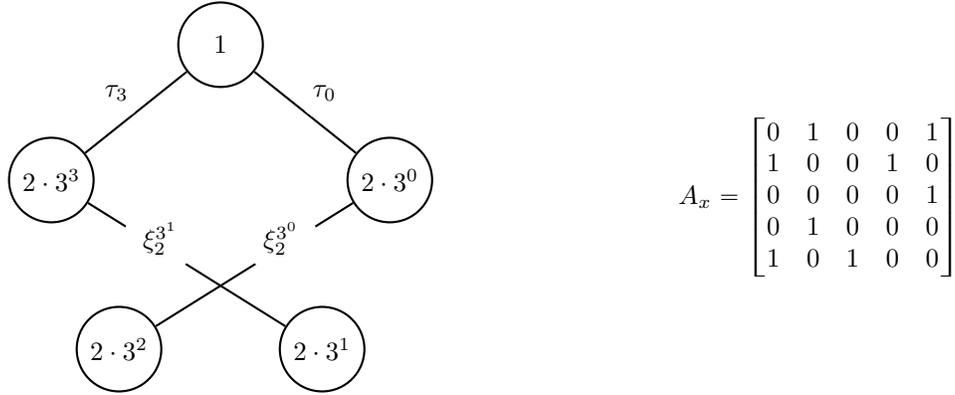
        \begin{minipage}[c]{.55\textwidth}
            \begin{tikzpicture}[scale=.9]
                \begin{scope}[every node/.style={circle,thick,draw}]
                    \node (1) at (0,2) {$\phantom{m}1\phantom{m}$};
                    \node (20) at (2.5,0) {$2\cdot 3^0$};
                    \node (21) at (1.5,-2.5) {$2\cdot 3^1$};
                    \node (22) at (-1.5,-2.5) {$2\cdot 3^2$};
                    \node (23) at (-2.5,0) {$2\cdot 3^3$};
                \end{scope}
                \begin{scope}[>={Stealth[black]},
                              every node/.style={fill=white,circle},
                              every edge/.style={draw, thick}]
                    \path [above right] (1) edge node {$\tau_0$} (20);
                    \path [above left] (1) edge node {$\tau_3$} (23);
                    \path [above right] (20) edge node {\color{black} $\xi_2^{3^0}$} (22);
                    \path [above left] (21) edge node {\color{black} $\xi_2^{3^1}$} (23);
                \end{scope}
            \end{tikzpicture}
        \end{minipage}
        \hspace{1cm}
        \begin{minipage}[c]{.3\textwidth}
            $A_{x}=\begin{bmatrix}
                0 & 1 & 0 & 0 & 1 \\
                1 & 0 & 0 & 1 & 0 \\
                0 & 0 & 0 & 0 & 1 \\
                0 & 1 & 0 & 0 & 0 \\
                1 & 0 & 1 & 0 & 0 \\
            \end{bmatrix}$
        \end{minipage}
    \end{center}
\end{example}

\begin{example}[A graph with a Hamilton path]
    Consider the monomial $x = \tau_0 \tau_3 \xi_1^{13} \in A^\vee(3)$. Rewriting $x$ in standard form yields
   $$\tau_0 \tau_3  \xi_1^{3^2}\xi_1^{3^1}\xi_1^{3^0}  \in A^\vee(3)$$
   with associated graph and adjacency matrix depicted below. 
   \begin{center}
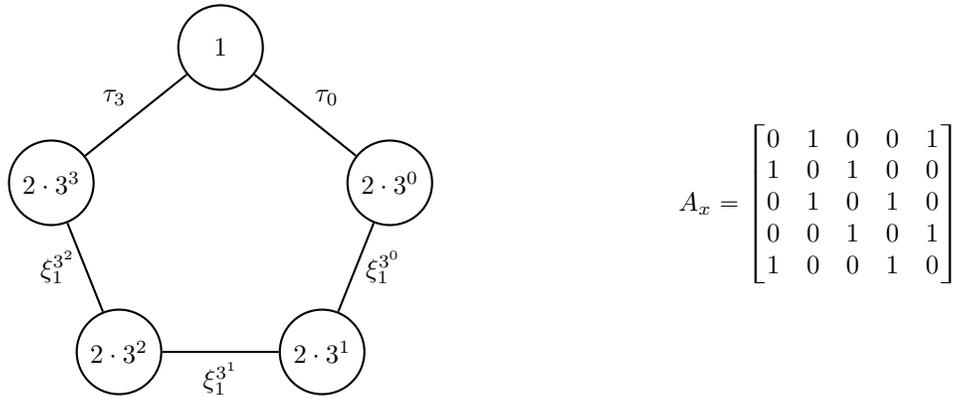
        \begin{minipage}[c]{.55\textwidth}
            \begin{tikzpicture}[scale=.9]
                \begin{scope}[every node/.style={circle,thick,draw}]
                    \node (1) at (0,2) {$\phantom{m}1\phantom{m}$};
                    \node (20) at (2.5,0) {$2\cdot 3^0$};
                    \node (21) at (1.5,-2.5) {$2\cdot 3^1$};
                    \node (22) at (-1.5,-2.5) {$2\cdot 3^2$};
                    \node (23) at (-2.5,0) {$2\cdot 3^3$};
                \end{scope}
                \begin{scope}[>={Stealth[black]},
                              every node/.style={},
                              every edge/.style={draw, thick}]
                    \path [above right] (1) edge node {$\tau_0$} (20);
                    \path [above left] (1) edge node {$\tau_3$} (23);
                    \path [right] (20) edge node {\color{black} $\xi_1^{3^0}$} (21);
                    \path [below] (21) edge node {\color{black} $\xi_1^{3^1}$} (22);
                    \path [left] (22) edge node {\color{black} $\xi_1^{3^2}$} (23);
                \end{scope}
            \end{tikzpicture}
        \end{minipage}
        \hspace{1cm}
        \begin{minipage}[c]{.3\textwidth}
            $A_{x}=\begin{bmatrix}
                0 & 1 & 0 & 0 & 1 \\
                1 & 0 & 1 & 0 & 0 \\
                0 & 1 & 0 & 1 & 0 \\
                0 & 0 & 1 & 0 & 1 \\
                1 & 0 & 0 & 1 & 0 \\
            \end{bmatrix}$
        \end{minipage}
    \end{center}
\end{example}

\begin{example}[Viewing the graph as a directed graph]
    Consider the monomial 
    \[
    x = \tau_0 \tau_3 \xi_1^{13} \in A^\vee(3).
    \]
    Rewriting $x$ in standard form yields
   $$\tau_0 \tau_3 \xi_1^{3^2}\xi_1^{3^1}\xi_1^{3^0}  \in A^\vee(3)$$
   with associated graph and adjacency matrix depicted below. 
   \begin{center}
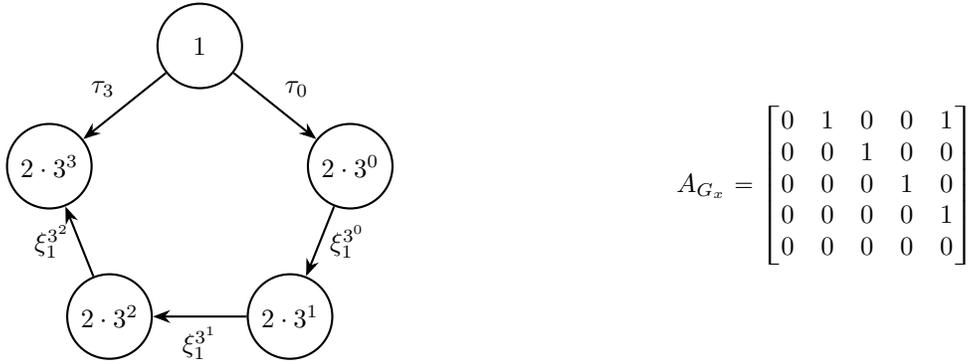
        \begin{minipage}[c]{.55\textwidth}
            \begin{tikzpicture}[scale=.8]
                \begin{scope}[every node/.style={circle,thick,draw}]
                    \node (1) at (0,2) {$\phantom{m}1\phantom{m}$};
                    \node (20) at (2.5,0) {$2\cdot 3^0$};
                    \node (21) at (1.5,-2.5) {$2\cdot 3^1$};
                    \node (22) at (-1.5,-2.5) {$2\cdot 3^2$};
                    \node (23) at (-2.5,0) {$2\cdot 3^3$};
                \end{scope}
                \begin{scope}[>={Stealth[black]},
                              every node/.style={},
                              every edge/.style={draw, thick, ->}]
                    \path [above right] (1) edge node {$\tau_0$} (20);
                    \path [above left] (1) edge node {$\tau_3$} (23);
                    \path [right] (20) edge node {\color{black} $\xi_1^{3^0}$} (21);
                    \path [below] (21) edge node {\color{black} $\xi_1^{3^1}$} (22);
                    \path [left] (22) edge node {\color{black} $\xi_1^{3^2}$} (23);
                \end{scope}
            \end{tikzpicture}
        \end{minipage}
        \hspace{1cm}
        \begin{minipage}[c]{.3\textwidth}
            $A_{G_x}=\begin{bmatrix}
                0 & 1 & 0 & 0 & 1 \\
                0 & 0 & 1 & 0 & 0 \\
                0 & 0 & 0 & 1 & 0 \\
                0 & 0 & 0 & 0 & 1 \\
                0 & 0 & 0 & 0 & 0 \\
            \end{bmatrix}$
        \end{minipage}
    \end{center}
    In general, with the vertex set being ordered we can always consider the graphs to be directed graphs. The monomial information can be read from both the directed and undirected graphs. Notice that when the graph \textit{is} viewed as a directed graph, the adjacency matrix is upper triangular and no longer symmetric (for $l>k$, $a_{l,k}=0$). 
\end{example}

\subsection{Connectedness and Unilaterality} \label{sec:Connect&UnilatOdd}

We begin by considering the subgraph $G_x \backslash \{1\}$ given by omitting the vertex $1$. This subgraph has adjacency matrix 
\begin{equation} \label{eq:AdjMpOdd}
\begin{bmatrix}
    0 & c_{1,0} & c_{2,0} & & c_{n-1,0} & c_{n,0} \\
    c_{1,0} & 0 & c_{1,1} & & c_{n-2,1} & c_{n-1,1} \\
    c_{2,0} & c_{1,1} & 0 & & c_{n-3,2} & c_{n-2,2} \\
    &  &  & \ddots & \\
    c_{n-1,0} & c_{n-2,1} & c_{n-3,2} & & 0 & c_{1,n-1} \\
    c_{n,0} & c_{n-1,1} & c_{n-2,2} & & c_{1,n-1} & 0 
\end{bmatrix}.
\end{equation}
Noticing that this matrix has nearly the same form as the adjacency matrix in the $p = 2$ case (\ref{eq:AdjMp=2}), we can apply a similar argument to prove the following proposition. 
\begin{proposition}
    The subgraph $G_x \backslash \{1\}$ is connected if and only if the integers
    \[
    C(r,s) : = \sum_{t = 1}^n \sum_{T_t'} \prod_{k = 1}^t c_{|b_k-b_{k-1}|, \max (b_{k-1}, b_k)-1} 
    \]
    are positive for all integers $r$ and $s$ with $0 \leq r < s \leq n$. 
\end{proposition}

Returning to consider the entire graph $G_x$, we prove the following theorem. 
\begin{theorem} \label{thm:connectedOdd}
    The graph $G_x$ is connected if and only if 
    \begin{enumerate}
        \item the integers 
        \[
        C(r,s) : = \sum_{t = 1}^n \sum_{T_t'} \prod_{k = 1}^t c_{|b_k-b_{k-1}|, \max (b_{k-1}, b_k)-1} 
        \]
        are positive for all integers $r$ and $s$ with $0 \leq r < s \leq n$ and
        \item the sum 
        \[
        \sum_{k=0}^n \epsilon_k > 0.
        \]
    \end{enumerate}
\end{theorem}

\begin{proof}
    The first condition guarantees that the subgraph $G \backslash \{ 1 \}$ is connected. The second condition guarantees that the vertex $1$ is connected by an edge to the subgraph $G \backslash \{ 1\}$.
\end{proof}

Similarly to the case where $p = 2$, the graphs $G_x$ can also naturally be viewed as directed graphs with edges oriented in the direction of the larger vertex. We will again denote the directed version of the graph $G_x$ by $G^\textup{dir}_x.$ Similar techniques of proof yield the following characterization of the unilaterality of $G_x$.

\begin{theorem} \label{thm:unilateralOdd}
    The directed graph $x^{\text{dir}}$ is unilateral if and only if 
    \begin{enumerate}
        \item the integers
        \[
        U(r,s) := \sum_{t=1}^{n}\sum_{{T}''_t}\prod_{k=1}^t c_{b_{k-1},b_k}
        \]
        are positive for all integers $r,s$ with $0 \leq r < s \leq n + 1$, where ${T}_t''$ is the set of all $(t + 1)$-tuples $(b_0, b_1, \cdots, b_t)$ with $b_i \in \{0, 1, \cdots, n + 1\}$ and $r = b_0 < b_1 < \cdots < b_t = s$,
        \item and furthermore $x$ is divisible by $\tau_0$.
    \end{enumerate}
\end{theorem}

\subsection{Trees and Hamiltonian cycles} \label{sec:Trees&HamiltonOdd}

\begin{theorem} \label{thm:treesOdd}
    Suppose $x$ is a monic monomial in $A^\vee(n)$ such that $G_x$ is connected. Then $G_x$ is a tree if and only if
   \begin{enumerate}
       \item the sum \[
            \sum_{i=1}^{n} \alpha_p(r_i) = n,
            \]
            where \( \alpha_p(r_i) \) denotes the number of nonzero digits in the base-\(p\) expansion of \( r_i \), and
            \item exactly one \( \tau_k \) divides the monic monomial \( x \), for some \( 0 \leq k \leq n \).
   \end{enumerate}
\end{theorem}

\begin{proof}
Let $x$ be a monic monomial in $A^\vee(n)$ such that $G_x$ is connected. Recall that in our graph construction, a digit equal to $1$ in the base p expansion of $r_i$ corresponds to having exactly one edge between two vertices of the graph. Thus, \[
\sum_{i=1}^{n} \alpha_p(r_i) = n
\] implies that there are exactly n edges contained in the subgraph $G_x \backslash \{1 \}$. If additionally, a singular \( \tau_k \) divides $x$ with \( 0 \leq k \leq n \), $x$ is connected but not cyclic because $G_x$ consists of $n+1$ distinct edges in a graph with $n+2$ vertices. Thus the graph $G_x$ is a tree. 
\end{proof}

Recall that a Hamiltonian directed path for a graph G is defined as a directed path containing every vertex of G. 

\begin{theorem}
Let $x$ be a monic monomial in $A^\vee (n)$. The graph $G_x$ has a Hamiltonian-directed path if and only if $\tau_0 \xi_1^{\frac{p^{n}-1}{p-1}}$ divides $x$.
\end{theorem}

\begin{proof} Suppose the graph $G^\textup{dir}_x$ associated with $x$ $\in$ $A^\vee(n)$ has a Hamiltonian directed path. Since the tail of each edge in $G_x^\textup{dir}$ is less than the head, and every vertex appears in the Hamiltonian directed path, the path must proceed through the ordered vertex set $V_x = \{1, 2, 2p, 2p^2, \cdots, 2p^n\}$ in order. In order for there to be an edge between the vertices $1$ and $2$, $\tau_0$ must divide $x$. In order for there to be an edge between the vertices $2p^{k}$ and $2p^{k + 1}$, $\xi_1^{{p^k}}$ must divide $x$. Hence 
\[
\tau_0 \xi_1^{\sum_{k = 0}^n p^k} = \tau_0 \xi_1^{\frac{p^{n}-1}{p-1}}
\]
must divide $x$. On the other hand, if $\tau_0 \xi_1^{\frac{p^{n}-1}{p-1}}$ divides $x$, the graph $G_x^\textup{dir}$ has a Hamiltonian-directed path by construction.  
\end{proof}

Next we turn to establishing a criteria for when the graph $G_x$ has a Hamilton cycle. Our goal will be to make use of Dirac's theorem, which says that a simple graph with $n$ vertices (${\displaystyle n\geq 3}$) is Hamiltonian if every vertex has degree ${\displaystyle {\tfrac {n}{2}}}$ or greater. Since we are working with multigraphs $G_x$, our first step will be to make an observation which allows us to reduce to the simple graph setting. 

Let $A_x$ be the adjacency matrix associated to $G_x$ with the $(i,j)$ entry of $A_x$ denoted $a_{i,j}$.
Define
    \[
    f_{\out} (2p^j) = \#\{1 < i \leq n | a_{j + 2, j + 2 + i} \geq 1\}
    \]
    and 
    \[
    f_{\inn} (2p^j) = \#\{1 < i \leq j+1 | a_{j + 2 -i, j + 2} \geq 1\}
    \] 
and let $G_x^\textup{R}$ denote the simple graph obtained from $G_x$ by replacing any multi-edges with a single edge between the same vertices. We will call $G_x^\textup{R}$ the reduced graph. Then the following proposition follows from the definitions of $A_x$ and $G_x^\textup{R}$.

\begin{proposition}  \label{lem:oddprimeInOutDeg}
    The out-degree of the vertex $2p^j$ in $G_x^\textup{R}$ considered as a directed graph is given by $f_{\out}$. The in-degree of the vertex $2p^j$ in $G_x^\textup{R}$ considered as a directed graph is given by $f_{\inn} (2p^j)$.
\end{proposition}

In order to guarantee a Hamiltonian cycle in $G_x$, it is sufficient to show there is a Hamiltonian cycle in $G_x^\textup{R}$. Since $f_{\inn}(2p^j) + f_{\out}(2p^j)$ gives the total degree of the vertices $2p^j$ ($0 \leq j \leq n$) in $G_x^\textup{R}$, it only remains to consider the degree of the vertex $1$ in order to apply Dirac's theorem. 

Since $1$ is the initial vertex in the directed graph $G_x^\textup{dir}$, the total degree of the vertex $1$ is \[
\sum_{k = 0}^n \epsilon_k.
\]
Then using Dirac's theorem we arrive at the following result. 

\begin{theorem} \label{thm:oddprimeHamilton}
    Suppose $x = \tau_0^{\epsilon_0} \cdots \tau_n^{\epsilon_n} \xi_1^{r_1} \cdots\ xi_n^{r_n} \in A^\vee(n)$. The corresponding graph $G_x$ has a Hamilton cycle if
    \begin{enumerate}
        \item $n>0$,
        \item for each $j$ with $0 \leq j \leq n$,
        \[
            f_{\inn}(2p^j) + f_{\out} (2p^j) \geq         \frac{n}{2}, 
        \]
        \item and \[\sum_{k=0}^n  \epsilon_k \geq \frac{n}{2}.\]
    \end{enumerate}  
\end{theorem}

\subsection{Graph theoretic interpretation of the coproduct} \label{sec:coprodOdd}

Similarly to Larson's graph theoretic description of the coproduct at the prime $p =2$ (\cref{thm:LarsonCoproduct}), we also have a graphical interpretation of the coproduct on elements $\xi_i^{p^j}$ and $\tau_i^{\epsilon_i} \in A^\vee (n)$ at odd primes. 
\begin{theorem} \label{thm:coproductOdd}
    The coproduct $\psi(\xi_i^{p^j}) \in A^\vee(n)$ is the sum of tensors of all pairs of edges that make length $2$ directed paths from the vertex $2 p^j$ to the vertex $2 p^{i +j}$ in the complete graph on the vertices $\{1, 2, 2p, 2p^2, \cdots, 2p^n \}$ considered as a directed graph.
     
    Similarly, coproduct $\psi(\tau_i) \in A^\vee(n)$ is the sum of tensors of all pairs of edges that make length $2$ directed paths from the vertex $1$ to the vertex $2 p^i$.
\end{theorem}

\begin{proof}
Our proof closely follows that of \cite[Theorem 1.5]{Larson22} in the $p = 2$ case. We first prove the statement for $x = \xi_i^{p^j}$. Recall that the coproduct on the dual Steenrod algebra $A^\vee$ is given by
\begin{equation} \label{eq:coprod}
\psi(\xi_n) = \xi_n \otimes 1 + 1 \otimes \xi_n + \sum_{k = 1}^{n - 1} \xi_{n - k}^{p^k}.
\end{equation}
The first two summands on the right-hand side of (\ref{eq:coprod}) represent degenerate length $2$ directed paths from $2 p^j$ to $2 p^{i + j}$. A non-degenerate length $2$ directed path from $2 p^j$ to $2 p^{i + j}$ corresponds to a choice of an intermediate vertex, which is of the form $2 p^{j + k}$ for some $k$ where $1 \leq k \leq i - 1$. Given this choice of $k$, the edge from $2p^j$ to $2p^{j + k}$ is $\xi_{j + k - j}^{p^j} = \xi_k^{p^j}$ and the edge from $2p^{j + k}$ to $2 p^{i + j}$ is $\xi_{i + j - ( j + k)}^{p^{i + k}}$. The terms of the sum indexed by $k$ on the far-right side of (\ref{eq:coprod}) correspond precisely to the pairs of edges just described. 
Similarly, 
\[
\psi(\tau_n) = \tau_n \otimes 1 + 1 \otimes \tau_n + \sum_{k = 0}^{n - 1} \xi_{n - k}^{p^k} \otimes \tau_k
\]
and a similar argument gives the desired statement.
\end{proof} 

\begin{example}
Suppose $p = 3$ and consider $\xi_2 \in A^\vee(3)$. Then  
\[
\psi(\xi_2) = \textcolor{darkblue'}{\xi_2 \otimes 1} + \textcolor{orange'}{\xi_1^3 \otimes \xi_1} + \textcolor{green'}{1 \otimes \xi_2}
\]
as depicted in \cref{fig:xi2Coprod}. Note that the loops in \cref{fig:xi2Coprod} denote degenerate paths corresponding to the factors of $1$ appearing in the tensor product.  The dashed arrows are directed edges between vertices that do not form directed paths of length $2$ between the vertices $2\cdot3^0$ and $2 \cdot 3^2$ in the complete graph on the vertices $\{1,\, 2 \cdot 3^0,\, 2 \cdot 3, \, 2 \cdot 3^2,\, 2 \cdot 3^3 \}$ considered as a directed graph. The solid arrows are edges that do form directed paths of length $2$ between the vertices $2\cdot3^0$ and $2 \cdot 3^2$.\color{black}

\begin{figure}[b]
\centering
\begin{tikzpicture}[->, thick, scale=1.5,
    every node/.style={circle, draw, minimum size=0.8cm},
    blueedge/.style={->, thick, blue},
    greenedge/.style={->, thick, green!80!black},
    rededge/.style={->, thick, red},
    dottededge/.style={->, dotted, thick}
    ]

\node (1) at (90:2) {$\phantom{m}1 \phantom{m}$};
\node (2) at (18:2) {$2 \cdot 3^0$};
\node (6) at (306:2) {$2 \cdot 3^1$};
\node (18) at (234:2) {$2 \cdot 3^2$};
\node (54) at (162:2) {$2 \cdot 3^3$};

\draw[green', bend right=15] (2) to (18);
\draw[orange'] (2) to (6);
\draw[orange'] (6) to (18);
\draw[color=darkblue', bend left=15] (2) to (18);

\path[green'] (2) edge[loop right] ();
\path[darkblue'] (18) edge[loop left] ();

\draw[dottededge] (1) -- (54);
\draw[dottededge] (2) -- (54);
\draw[dottededge] (6) -- (54);
\draw[dottededge] (18) -- (54);
\draw[dottededge] (1) -- (2);
\draw[dottededge] (6) -- (54);
\draw[dottededge] (1) -- (6);
\draw[dottededge] (1) -- (18);

\end{tikzpicture}
\caption{Coproduct construction of $\xi_2$ in $A^\vee_3(3)$}
\label{fig:xi2Coprod}
\end{figure}

\end{example}
\begin{example}
Suppose $p = 3$ and consider $\tau_2 \in A^\vee(3)$. Then 
\[
\psi(\tau_2) = \textcolor{darkblue'}{\tau_2 \otimes 1} + \textcolor{orange'}{\xi_2 \otimes \tau_0} + \textcolor{green'}{\xi_1^3 \otimes \tau_1} + \textcolor{lightblue'}{1 \otimes \tau_2}
\] 
as depicted in \cref{fig:tau2Coprod}. Again loops denote degenerate paths corresponding to factors of $1$ appearing in the tensor product.

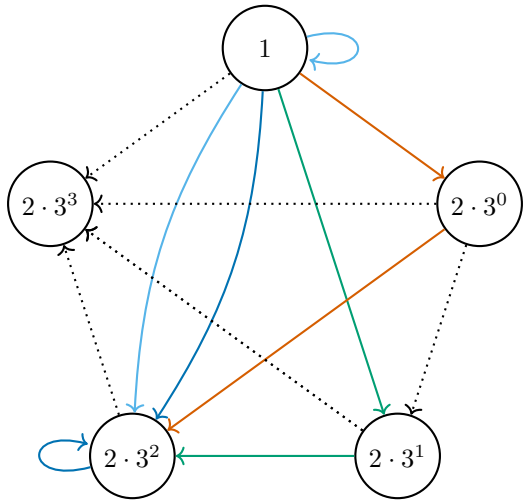
\begin{figure}[b]
\centering
\begin{tikzpicture}[->, thick, scale=1.5,
    every node/.style={circle, draw, minimum size=0.8cm},
    blueedge/.style={->, thick, blue},
    greenedge/.style={->, thick, green!80!black},
    rededge/.style={->, thick, red},
    cyanedge/.style={->, thick, cyan},
    dottededge/.style={->, dotted, thick}
    ]

\node (1) at (90:2) {$\phantom{m}1 \phantom{m}$};
\node (2) at (18:2) {$2 \cdot 3^0$};
\node (6) at (306:2) {$2 \cdot 3^1$};
\node (18) at (234:2) {$2 \cdot 3^2$};
\node (54) at (162:2) {$2 \cdot 3^3$};

\draw[green'] (1) to (6);
\draw[green'] (6) to (18);
\draw[orange'] (1) to (2);
\draw[orange'] (2) to (18);
\draw[darkblue', bend left=15] (1) to (18);
\draw[lightblue', bend right=15] (1) to (18);

\path[darkblue'] (18) edge[loop left] ();
\path[lightblue'] (1) edge[loop right] ();

\draw[dottededge] (1) -- (54);
\draw[dottededge] (2) -- (54);
\draw[dottededge] (6) -- (54);
\draw[dottededge] (18) -- (54);
\draw[dottededge] (6) -- (54);
\draw[dottededge] (2) -- (6);

\end{tikzpicture}

\caption{Coproduct construction of $\tau_2$ in $A^\vee_3(3)$}
\label{fig:tau2Coprod}
\end{figure}
\end{example}

\subsection{Graph theoretic interpretation of the antipode} 
Similarly to Larson’s graph theoretic description of the antipode at the prime $p = 2$ (\cref{thm:LarsonCoproduct}), we also have a graphical interpretation of the antipode $c: A^\vee (n) \to A^\vee (n)$ on the elements $\xi_i^{p^j}$ and $\tau_i$ at odd primes. 

\begin{theorem} \label{thm:AntipodeOdd} 
    The antipode $c(\xi_i^{p^j}) \in A^\vee(n)$ is the signed sum of all directed paths from the vertex $2 p^j$ to the vertex $2 p^{i + j}$ in the complete graph on the vertices $\{1,2, 2p, 2p^2, \cdots, 2p^n\}$ considered as a directed graph. The positive terms in the sum correspond exactly to paths consisting of an even number of edges from the vertex $2$ to the vertex $2p^{i +j}$.

    Similarly, the antipode $c(\tau_i) \in A^\vee(n)$ is the signed sum of all directed paths from the vertex $1$ to the vertex $2 p^{i}$. The positive terms correspond exactly to paths consisting of an even number of edges from the vertex $1$ to the vertex $2 p^i$. \color{black}
\end{theorem} 

\begin{proof}
Recall that the antipode map $c: A^\vee (n) \to A^\vee (n)$ of the Hopf algebra structure sends
\begin{align*}
    c(\xi_n) & = - \xi_n - \sum_{k = 0}^{n - 1} \xi_{n - k}^{p^k} c (\xi_k),
\end{align*}
and
\begin{align*}
    c(\tau_n) & = - \tau_n - \sum_{k = 0}^{ n - 1} \xi_{n - k}^{p^k} c(\tau_k).
\end{align*}
We will prove the statement for $c(\xi_i^{p^j})$ by induction on $i$. The proof for $c(\tau_i)$ is similar. 

To begin the induction, consider $c(\xi_1^{p^j}) = c(\xi_1)^{p^j} = (-\xi_1)^{p^j} = - \xi_1^{p^j}$ since $p$ is an odd prime. In the  complete graph on the vertices $\{1,2, 2p , 2p^2, \cdots, 2p^n\}$, the monomial $\xi_1^{p^j}$ corresponds exactly to the a path from $2p^j$ to $2p^{j + 1}$.\color{black} This is a path of length one so we observe the desired statement is true when $i = 1$. 

Now suppose $c(\xi_\ell^{p^j})$ is the signed sum of all directed paths from $2{p^j}$ to $2 p^{\ell + j}$ for all $\ell \leq i$ where paths of an even length correspond to positive terms and paths of an odd length correspond to negative terms. We must show that 
\begin{equation} \label{eq:conj}
c(\xi_{i + 1}^{p^j}) = - \xi_{i + 1}^{p^j} - \sum_{k = 0}^{i} \xi_{i + 1 - k}^{p^{k + j}} c(\xi_k^{p^j})
\end{equation}
is the signed sum of all directed paths from $2p^j$ to $2p^{i + 1 + j}$ where directed paths of an even length correspond to positive terms and directed paths of an odd length correspond to negative terms.

Consider the term $- \xi_{i + 1 - k}^{p^{k + j}} c(\xi_k^{p^j})$ in \cref{eq:conj}. The $c(\xi_k^{p^j})$ factor in this term consists of a signed sum of all paths from $2 p^j$ to $2 p^{j + k}$ while the $\xi_{i + 1 - k}^{p^{k + j}}$ factor consists of the edge from $2 p^{j + k}$ to $2 p^{j + i +1}$. Thus $\xi_{i + 1 - k}^{p^{k + j}} c(\xi_k^{p^j})$ is a signed sum of all directed paths from $2 p^j$ to $2 p^{j + k}$ followed by the edge from $2 p^{j + k}$ to $2 p^{j + i +1}$ in the reduced graph $G_x^\textup{R}$. Observe that the length of each directed path in $-\xi_{i + 1 - k}^{p^{k + j}} c(\xi_k^{p^j})$ is one step longer than in $c(\xi_k^{p^j})$. Moreover, the sign in front of each term appearing in $-\xi_{i + 1 - k}^{p^{k + j}} c(\xi_k^{p^j})$ is the opposite of that appearing in $c(\xi_k^{p^j})$. Hence directed paths of even length still correspond to positive terms in the sum and directed paths of odd length still correspond to negative terms in the sum.
Thus summing over all $k$ where $0 \leq k \leq i$ gives the desired signed sum of all paths from $2 p^j$ to $2 p^{i + 1 + j}$ passing through at least one intermediate vertex. The term $-\xi_{i + 1}^{p^j}$ gives the remaining path from $2 p^j$ to $2 p^{i + 1 +j}$ passing through no intermediate vertices. Hence $c(\xi_{i + 1}^{p^j})$ is the signed sum of all paths from $2p^j$ to $2 p^{i + 1 + j}$ where directed paths of an even length correspond to positive terms and directed paths of an odd length correspond to negative terms.

\end{proof}

\begin{example}
Let $p$ be an odd prime and consider $\xi_2 \in A^\vee(3)$. Then
\[
c(\xi_2) = - \textcolor{orange'}{\xi_2} + \textcolor{darkblue'}{\xi_1^{p + 1}}
\] 
as depicted in \cref{fig:xi2conj}.

\begin{figure}[ht]
\centering
\begin{tikzpicture}[->, thick, scale=1.5,
    every node/.style={circle, draw, minimum size=0.8cm},
    blueedge/.style={->, thick, blue},
    greenedge/.style={->, thick, green!80!black},
    rededge/.style={->, thick, red},
    orangeedge/.style={->, thick, orange},
    dottededge/.style={->, dotted, thick}
    ]

\node (1) at (90:2) {$\phantom{m}1 \phantom{m}$};
\node (2) at (18:2) {$2 \cdot p^0$};
\node (6) at (306:2) {$2 \cdot p^1$};
\node (18) at (234:2) {$2 \cdot p^2$};
\node (54) at (162:2) {$2 \cdot p^3$};

\begin{scope}[>={Stealth[black]},
                              every node/.style={},
                              every edge/.style={draw, thick}]
                    \path [above left,orange'] (2) edge node {\color{orange'} $\xi_2$} (18);
                    \path [darkblue',right] (2) edge node {\color{darkblue'} $\xi_1$} (6);
                    \path [darkblue', below] (6) edge node {\color{darkblue'} $\xi_1^p$} (18);
\end{scope}

\draw[dottededge] (1) -- (54);
\draw[dottededge] (2) -- (54);
\draw[dottededge] (6) -- (54);
\draw[dottededge] (18) -- (54);
\draw[dottededge] (6) -- (54);
\draw[dottededge] (1) -- (2);
\draw[dottededge] (1) -- (18);
\draw[dottededge] (1) -- (6);

\end{tikzpicture}

\caption{Antipode of $\xi_2$ in $A^\vee_3(3)$}
\label{fig:xi2conj}
\end{figure}
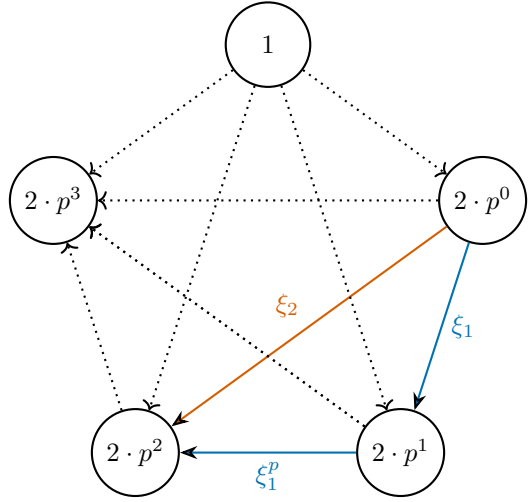
\end{example}

\begin{example}
Let $p$ be an odd prime and consider $\tau_2 \in A^\vee(3)$. Then 
\[
c(\tau_2) = - \textcolor{orange'}{\tau_2} + \textcolor{green'}{\tau_1 \xi_1^p} + \textcolor{pink'}{\tau_0\xi_2} - \textcolor{lightblue'}{\tau_0\xi_1^{p + 1}}
\] as depicted in \cref{fig:tau2Anti}.

\begin{figure}[ht]
\centering
\begin{tikzpicture}[->, thick, scale=1.5,
    every node/.style={circle, draw, minimum size=0.8cm},
    blueedge/.style={->, thick, pink'},
    greenedge/.style={->, thick, green'},
    rededge/.style={->, thick, orange'},
    cyanedge/.style={->, thick, lightblue'},
    dottededge/.style={->, dotted, thick}
    ]

\node (1) at (90:2) {$\phantom{m}1 \phantom{m}$};
\node (2) at (18:2) {$2 \cdot p^0$};
\node (6) at (306:2) {$2 \cdot p^1$};
\node (18) at (234:2) {$2 \cdot p^2$};
\node (54) at (162:2) {$2 \cdot p^3$};

\begin{scope}[>={Stealth[black]},
                              every node/.style={},
                              every edge/.style={draw, thick}]
    \path [above left,orange'] (1) edge node {\color{orange'} $\tau_2$} (18);
    \path [green',right] (1) edge node {\color{green'} $\tau_1$} (6);
    \path [green', below, bend left = 15] (6) edge node {\color{green'} $\xi_1^p$} (18);
    \path [lightblue', below left, bend right = 15] (1) edge node {\color{lightblue'} $\tau_0$} (2);
    \path [lightblue', right] (2) edge node {\color{lightblue'} $\xi_1$} (6);
    \path [lightblue', bend right = 15, above] (6) edge node {$\xi_1^p$} (18);
    \path [pink', bend left = 15, above right] (1) edge node {\color{pink'} $\tau_0$} (2);
    \path [pink', above left] (2) edge node {\color{pink'} $\xi_2$} (18);
\end{scope}

\draw[dottededge] (1) -- (54);
\draw[dottededge] (2) -- (54);
\draw[dottededge] (6) -- (54);
\draw[dottededge] (18) -- (54);
\draw[dottededge] (6) -- (54);

\end{tikzpicture}

\caption{Antipode of $\tau_2$ in $A^\vee(3)$}
\label{fig:tau2Anti}
\end{figure}
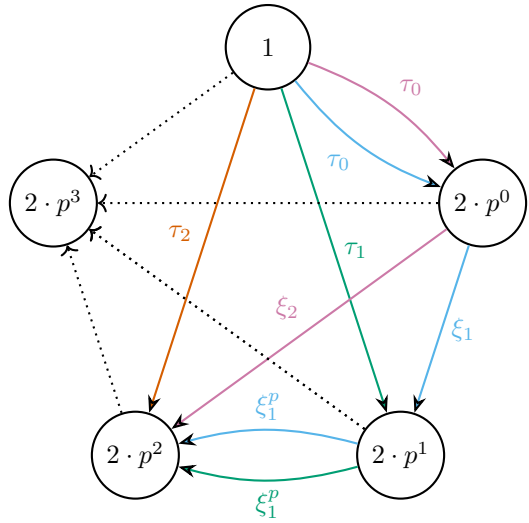
\end{example}

\begin{cor} \label{cor:unilateralityOdd}
    For $x \in A^\vee (n)$, the directed graph $G_x^{\textup{dir}}$ is unilateral if and only if
    \begin{enumerate}
        \item for each $\xi_i^{p^j} \in A^\vee (n)$, at least one summand of $c(\xi_i^{p^i})$ is a factor of $x$
        \item and $\tau_0$ is a factor of $x$.
    \end{enumerate}
\end{cor}

\begin{proof}
    Note that the directed graph $G_x^\textup{dir}$ is unilateral if and only if there is a directed path connecting any two of its vertices, say $2 p^{i}$ and $2 p^{ i + j}$ or $1$ and $2 p^{i + j}$. \cref{thm:AntipodeOdd} shows this is equivalent to the demand that at least one summand of $c(\xi_i^{p^j})$ and $\tau_0$ appear as factors of $x$. 
\end{proof}

\section{A $C_2$-equivariant graphical construction} \label{sec:C2}

\subsection{The coefficients $\mM_2$} \label{sec:M2} Before describing the $C_2$-equivariant dual Steenrod algebra, we give a more detailed description of the bigraded coefficient ring 
\[
\mM_2 \cong \mF_2[a,u] \oplus \frac{\mF_2[a,u]}{(a^\infty, u^\infty)} \{ \theta \},
\]
where $|a|=(-1,-1)$, $|u|=(0,-1)$. A bigraded plot of $\mM_2$ appears in \cref{fig:M2} where each dot represents a copy of $\mF_2$. The topological dimension is given on the horizontal axis while the vertical axis is the weight. 

\begin{figure}[h!]
    \centering
    \resizebox{2.25in}{2in}{\begin{tikzpicture}
        \draw[<->,color=gray] (-5, -5) -- (5, -5) node[scale=0.7, right,yshift=-0.2in] {}; 
		\foreach \x/\xtext in {-4,-3,-2,-1,0,1,2,3,4}
		\draw[shift={(\x,-5)},color=gray] (0pt,2pt) -- (0pt,-2pt) node[scale=0.75,below,color=gray] {$\xtext$}; 
		
		\draw[<->,color=gray] (-6, -5) -- (-6, 5) node[scale=0.7, above left, color=gray] {}; 
		\foreach \y/\ytext in {-4,-3,-2,-1,0,1,2,3,4}
		\draw[shift={(-6,\y)}, color=gray] (2pt,0pt) -- (-2pt,0pt) node[scale=0.75,left, color=gray] {$\ytext$}; 
			
		\node at (-4,-4) {\textup{$\bullet$}};
		\node at (-3,-4) {\textup{$\bullet$}};
		\node at (-2,-4) {\textup{$\bullet$}};
		\node at (-1,-4) {\textup{$\bullet$}};
		\node at (0,-4) {\textup{$\bullet$}};
		
		\node at (-3,-3) {\textup{$\bullet$}};
		\node at (-2,-3) {\textup{$\bullet$}};
		\node at (-1,-3) {\textup{$\bullet$}};
		\node at (0,-3) {\textup{$\bullet$}};
		
		\node at (-2,-2) {\textup{$\bullet$}};
		\node at (-1,-2) {\textup{$\bullet$}};
		\node at (0,-2) {\textup{$\bullet$}};
		
		\node at (-1,-1) {\textup{$\bullet$}};
		\node[scale=1.5] at (-1.4,-0.9) {\textup{$a$}};
		\node at (0,-1) {\textup{$\bullet$}};
		\node[scale=1.5] at (0.4,-0.9) {\textup{$u$}};
		
		\node at (0,0) {\textup{$\bullet$}};
		\node at (0.4,0.15) {\textup{$1$}};
		
		\node at (0,2) {\textup{$\bullet$}};
		\node[scale=1.5] at (-0.3,2) {\textup{$\theta$}};
		
		\node at (0,3) {\textup{$\bullet$}};
		\node at (1,3) {\textup{$\bullet$}};
		
		\node at (0,4) {\textup{$\bullet$}};
		\node at (1,4) {\textup{$\bullet$}};
		\node at (2,4) {\textup{$\bullet$}};
		
		\node at (0,5) {\textup{$\bullet$}};
		\node at (1,5) {\textup{$\bullet$}};
		\node at (2,5) {\textup{$\bullet$}};
		\node at (3,5) {\textup{$\bullet$}};
    \end{tikzpicture}}
    \caption{The coefficients $\mM_2$}
    \label{fig:M2}
\end{figure}

\begin{remark} \rm
    The coefficients $\mM_2$ are the $RO(C_2)$-graded homology of a point with coefficients in the $C_2$-constant Mackey functor $\umF_2$. 
\end{remark}

\subsection{The $C_2$-equivariant dual Steenrod algebra}
{ Recall that the $C_2$-equivariant dual Steenrod algebra is given by 
$$
A_{C_2}^\vee =\bigslant{\mathbb{M}_2[\xi_1,\xi_2,...,\tau_0,\tau_1,...]}{(\tau_i^2=(u+a\tau_0)\xi_{i+1}+a\tau_{i+1})},
$$
where |$\xi_i|=\rho(2^i-1)$ and $|\tau_i|=2^i\rho-\sigma$ (see \cite[Theorem 6.41]{HuKriz01} or \cite[Theorem 2.14]{LiShiWAngXu19} for the full Hopf algebroid structure). Also recall that given an integer $n \geq 0$ the quotient $A_{C_2}^\vee  (n)$ is 
\[
A_{C_2}^\vee(n) \cong \bigslant{\mathbb{M}_2[\xi_1,\xi_2,..., \xi_n, \tau_0,\tau_1,..., \tau_n]}{(\xi_i^{2^{n-i+1}},\,\tau_i^2=(u+a\tau_0)\xi_{i+1}+a\tau_{i+1})}.
\]}

\subsection{A graphical interpretation} Given a monomial $x \in A_{C_2}^\vee (n)$ we can use the relation $\tau_i^2=(u+a\tau_0)\xi_{i+1}+a\tau_{i+1}$ to write $x$ as a finite sum consisting of terms of the form
\begin{equation} \label{eq:C2terms}
w \tau_0^{\epsilon_0} \tau_1^{\epsilon_1} + \cdots + \tau_n^{\epsilon_n} \xi_1^{r_1} \xi_2^{r_2} \cdots \xi_n^{{ r_n}}
\end{equation}
where $w \in \mM_2$, $\epsilon_i \in \{0,1\}$ and $0 \leq r_i \leq 2^{n + 1 - i} - 1$. Because addition is represented as disjoint union in our graphical interpretation, we will focus on monic monomials of the form (\ref{eq:C2terms}) throughout this section. Similarly to Wood's construction in the nonequivariant setting at the prime $2$, we will also utilize the $2$-adic expansion of $r_i$,
\[
r_i = \sum_{m  = 0}^{n - i} c_{i, n - i-m} \cdot 2^{n - i - m}
\]
and extend Definitions \ref{sform1} and \ref{def:StandardFormOdd} to make the following definition.
\begin{deff} \label{def:StandardFormEquivariant}
    Given a monic monomial $x =\tau_0^{\epsilon_0} \tau_1^{\epsilon_1} \cdots  \tau_n^{\epsilon_n} \xi_1^{r_1} \xi_2^{r_2} \cdots \xi_n^{{ r_n}}$ where $\epsilon_i \in \{0,1\}$ and $0 \leq r_i \leq 2^{n + 1 - i} - 1$, the {\bf \textit{standard form}} of $x$ is 
    \[
     \tau_0^{\epsilon_0} \tau_1^{\epsilon_1} \cdots \tau_n^{\epsilon_n} \left(\prod_{i = 1}^n \prod_{m = 0}^{n - i} \xi_i^{c_{i, n - i -m} \cdot 2^{n - i - m}} \right)
    \]
    where the exponents on the $\xi_i$ are { expanded into} their $2$-adic expansion. 
\end{deff}

Given such a monic monomial $x \in A^\vee (n)$, we define the graph $G_x = (V_x, E_x) $ by setting the ordered vertex set $V_x = \{\sigma, \rho, 2 \rho, 2^2 \rho, \cdots 2^n \rho\}$ and the edge set
\[
E_x = \{\tau_i: \sigma \to 2^i \rho \}_{0 \leq i \leq n} \bigcup \left\{ \xi_i^{c_{i,j} \cdot 2^j} : 2^j \rho \to 2^{ i + j} \rho \right\}_{\substack{1 \leq i \leq n \\ 0 \leq j \leq n - i}}
\]
where each $\tau_i$ denotes an edge connecting vertices $\sigma$ and $2^i \rho$, and each $\xi_i^{c_{i,j} \cdot 2^j}$ denotes $c_{i,j}$ edges connecting vertices $2^j \rho$ and $2^{i + j} \rho$. Notice that, similarly to nonequivariant setting, $| \xi_i^{2^j} |$, the degree of the edge $\xi_i^{2^j}$ { when viewed as an element of $A^\vee_{C_2}(n)$} corresponds to the absolute value of the difference of the ends of $\xi_i^{2^j}$ and the degree of the edge $\tau_i$ { when viewed as an element of $A^\vee_{C_2}(n)$}, $| \tau_i |$, also corresponds to the absoslute value of the difference of the ends of $\tau_i$.  

Since $\mM_2$ is a bigraded ring, and the coefficient $w \in \mM_2$ may live in a nonzero bidegree, multiplication by $w \in \mM_2$ may change the degree of $\tau_0^{\epsilon_0} \tau_1^{\epsilon_1} \cdots \tau_n^{\epsilon_n} \xi_1^{r_1} \xi_2^{r_2} \cdots \xi_n^{r_n}$. To account for this in our graphical interpretation, we view the vertices $V_x$ of $G_x$ as embedded in $\mR^2$ with the embedding given by the { bidegree} of the vertices plus the { bidegree} of $w$ as illustrated in the next two examples.

\begin{example}
The monomial $x= \xi_1^3\xi_2\tau_0$ in $A_{C_2}^\vee(2)$ has associated graph $G_x$ as depicted in \cref{fig:C2graph}. Here, and throughout this section, we use the motivic grading convention where the topological degree is plotted along the horizontal axis while the weight is given by the vertical axis. 
\begin{center}
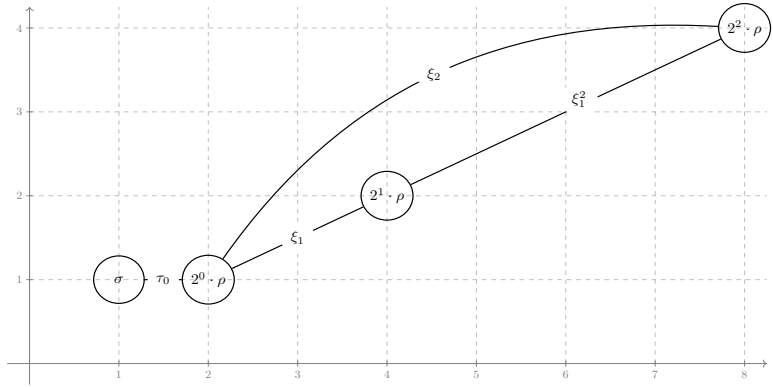
\begin{figure}[ht]
\resizebox{4in}{2in}{\begin{tikzpicture}
                \draw[step=2.0,gray!50!white,thin,dashed] (0,0) grid (16.5,8.5); 
                \draw[thin,gray,->] (0,-0.5) -- (0,8.5);  
                \draw[thin,gray,->] (-0.5,0) -- (16.5,0); 
            
                \foreach \x/\xtext in {1,2,3,4,5,6,7,8}
            		\draw[shift={(2*\x,0)},color=gray] (0pt,2pt) -- (0pt,-2pt) node[scale=0.75,below,color=gray] {$\x$}; 
            
                \foreach \y/\ytext in {1,2,3,4}
            		\draw[shift={(0,2*\y)}, color=gray] (2pt,0pt) -- (-2pt,0pt) node[scale=0.75,left, color=gray] {$\y$}; 
            
                \begin{scope}[every node/.style={circle,thick,draw}]
                    \node (11) at (2,2) {$\phantom{m}\sigma\phantom{m}$};
                    \node (20) at (4,2) {$2^0\cdot\rho$};
                    \node (21) at (8,4) {$2^1\cdot\rho$};
                    \node (22) at (16,8) {$2^2\cdot\rho$};
                \end{scope}
                \begin{scope}[>={Stealth[black]},
                            every edge/.style={draw,thick}]
                    \path (11) edge node [below] { $\tau_0$} (20);
                    \path [bend left] (20) edge node [above] { $\xi_2$} (22);
                    \path (20) edge node [below] { $\xi_1$} (21);
                    \path [above right] (21) edge node [below] { $\xi_1^2$} (22);
                \end{scope}
            \end{tikzpicture}}
    \caption{The graph $G_x$ associated to $x = \xi_1^3 \xi_2 \tau_0 \in A_{C_2}^\vee (2)$}
    \label{fig:C2graph}
\end{figure}
\end{center}
\end{example}

\begin{example}
    The monomial $x = a^2 u \xi_1^3 \xi_2 \tau_0 \in A_{C_2}^\vee(2)$ has associated graph $G_x$ as depicted in \cref{fig:C2graph2}. It may be helpful to recall that $| a | = (-1,-1)$ and $|u| = (0,-1)$.
    \begin{center}
    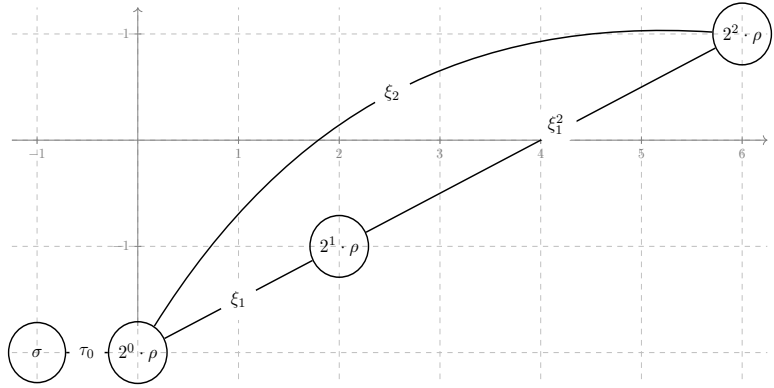
\begin{figure}
        \resizebox{4in}{2in}{\begin{tikzpicture}
                \draw[step=2.0,gray!50!white,thin,dashed] (-2.5,-4.5) grid (12.5,2.5); 
                \draw[thin,gray,->] (0,-3.5) -- (0,2.5);  
                \draw[thin,gray,->] (-2.5,0) -- (12.5,0); 
            
                \foreach \x/\xtext in {-1,1,2,3,4,5,6}
            		\draw[shift={(2*\x,0)},color=gray] (0pt,2pt) -- (0pt,-2pt) node[scale=0.75,below,color=gray] {$\x$}; 
            
                \foreach \y/\ytext in {-1,1}
            		\draw[shift={(0,2*\y)}, color=gray] (2pt,0pt) -- (-2pt,0pt) node[scale=0.75,left, color=gray] {$\y$}; 
            
                \begin{scope}[every node/.style={circle,thick,draw}]
                    \node (11) at (-2,-4) {$\phantom{m}\sigma\phantom{m}$};
                    \node (20) at (0,-4) {$2^0\cdot\rho$};
                    \node (21) at (4,-2) {$2^1\cdot\rho$};
                    \node (22) at (12,2) {$2^2\cdot\rho$};
                \end{scope}
                \begin{scope}[>={Stealth[black]},
                            every edge/.style={draw,thick}]
                    \path (11) edge node  [below] { $\tau_0$} (20);
                    \path [bend left] (20) edge node  [above] { $\xi_2$} (22);
                    \path (20) edge node  [below] { $\xi_1$} (21);
                    \path [above right] (21) edge node [below right] { $\xi_1^2$} (22);
                \end{scope}
            \end{tikzpicture}}
           
    \caption{The graph $G_x$ associated to $x = a^2 u \xi_1^3 \xi_2 \tau_0 \in A_{C_2}^\vee (2)$}
    \label{fig:C2graph2}
\end{figure}
\end{center}
\end{example}

Finally, the adjacency matrix $A_x$ associated to $G_x$ has exactly the same form as in the odd primary case (\ref{eq:AdjMatrixOdd}).
 
\subsection{Connectivity and Unilaterality} Similarly to the nonequivariant odd primes case (\textsection \ref{sec:Connect&UnilatOdd}), we begin by considering the subgraph $G_x \backslash \{ \sigma \}$. This subgraph has adjacency matrix 
\[
\begin{bmatrix}
    0 & c_{1,0} & c_{2,0} & & c_{n-1,0} & c_{n,0} \\
    c_{1,0} & 0 & c_{1,1} & & c_{n-2,1} & c_{n-1,1} \\
    c_{2,0} & c_{1,1} & 0 & & c_{n-3,2} & c_{n-2,2} \\
    &  &  & \ddots & \\
    c_{n-1,0} & c_{n-2,1} & c_{n-3,2} & & 0 & c_{1,n-1} \\
    c_{n,0} & c_{n-1,1} & c_{n-2,2} & & c_{1,n-1} & 0 
\end{bmatrix}.
\]
{ Since} this matrix has the same form as the adjacency matrix in the nonequivariant $p$ and odd prime case (\ref{eq:AdjMatrixOdd}), we can apply a similar argument to prove the following proposition. 
\begin{proposition}
    The subgraph $G_x \backslash \{\sigma\}$ is connected if and only if the integers
    \[
    C(r,s) : = \sum_{t = 1}^n \sum_{T_t'} \prod_{k = 1}^t c_{|b_k-b_{k-1}|, \max (b_{k-1}, b_k)-1} 
    \]
    are positive for all integers $r$ and $s$ with $0 \leq r < s \leq n$. 
\end{proposition}
Returning to consider the entire graph $G_x$, we prove the following theorem, which has the same statement as in the { nonequivariant odd-primary case} (\cref{thm:connectedOdd}). 
\begin{theorem} \label{thm:connectedC2}
     Suppose $x \in A_{C_2}^\vee(n)$ is a monic monomial { in standard form}. Then the graph $G_x$ is connected if and only if 
    \begin{enumerate}
        \item the integers 
        \[
        C(r,s) : = \sum_{t = 1}^n \sum_{T_t'} \prod_{k = 1}^t c_{|b_k-b_{k-1}|, \max (b_{k-1}, b_k)-1} 
        \]
        are positive for all integers $r$ and $s$ with $0 \leq r < s \leq n$ and
        \item the sum 
        \[
        \sum_{k=0}^n \epsilon_k > 0.
        \]
    \end{enumerate}
\end{theorem}

\begin{proof}
    The first condition guarantees that the subgraph $G \backslash \{ \sigma \}$ is connected. The second condition guarantees that the vertex $1$ is connected by an edge to the subgraph $G \backslash \{ \sigma \}$.
\end{proof}

Similarly to the nonequivariant setting, the graphs $G_x$ can naturally be viewed as directed graphs with edges oriented in the direction of the vertex of larger dimension. We will again denote the directed version of the graph $G_x$ by $G^\textup{dir}_x.$ Similar techniques of proof yield the following characterization of the unilaterality of $G_x$.

\begin{theorem} \label{thm:unilateralC2}
     Suppose $x \in A_{C_2}^\vee(n)$ is a monic monomial { in standard form}. Then the directed graph $G_x^\textup{dir}$ is unilateral if and only if
    \begin{enumerate}
        \item the integers
        \[
        U(r,s) := \sum_{t=1}^{n}\sum_{{T}''_t}\prod_{k=1}^t c_{b_{k-1},b_k}
        \]
        are positive for all integers $r,s$ with $0 \leq r < s \leq n + 1$, where ${T}_t''$ is the set of all $(t + 1)$-tuples $(b_0, b_1, \cdots, b_t)$ with $b_i \in \{0, 1, \cdots, n + 1\}$ and $r = b_0 < b_1 < \cdots < b_t = s$.
        \item And $x$ is divisible by $\tau_0$.
    \end{enumerate}
\end{theorem}

\subsection{Trees and Hamiltonian Cycles}

\begin{theorem}
    Suppose $x$ is a monic monomial { in standard form} in $A_{C_2}^\vee(n)$ such that $G_x$ is connected. Then $G_x$ is a tree if and only if
   \begin{enumerate}
       \item the sum \[
            \sum_{i=1}^{n} \alpha_p(r_i) = n,
            \]
            where \( \alpha_p(r_i) \) denotes the number of nonzero digits in the base-\(p\) expansion of \( r_i \), and
            \item exactly one \( \tau_k \) divides the connected graph \( x \), for some \( 0 \leq k \leq n \).
   \end{enumerate}
\end{theorem}

\begin{proof}
Let $x$ be a monic monomial { in standard form} in $A_{C_2}^\vee(n)$ such that $G_x$ is connected. Recall that in our graph construction, a digit equal to $1$ in the base 2 expansion of $r_i$ corresponds exactly to having one edge between two vertices of the graph. Thus, \[
\sum_{i=1}^{n} \alpha_p(r_i) = n
\] implies that there are exactly n edges contained in the subgraph $G_x \backslash \{\sigma \}$. If additionally, a singular \( \tau_k \) divides $x$ with \( 0 \leq k \leq n \), $x$ is connected but not cyclic because $G_x$ consists of $n+1$ distinct edges in a graph with $n+2$ vertices. Thus, the graph of $G_x$ is a tree. 
\end{proof}

Recall that a Hamiltonian directed path for a graph G is defined as a directed path containing every vertex of G. 

\begin{theorem}
Let $x$ be a monic monomial { in standard form} in $A_{C_2}^\vee (n)$. The graph $G_x$ has a Hamiltonian-directed path if and only if $\tau_0 \xi_1^{2^{n + 1}-1}$ divides $x$.
\end{theorem}

\begin{proof} Suppose the graph $G^\textup{dir}_x$ associated with $x$ $\in$ $A_{C_2}^\vee(n)$ has a Hamiltonian directed path. Because $G_x$ is directed such that each tail edge end vertex has representation dimension less than that of the head edge and every vertex appears in the Hamiltonian directed path, the path must proceed through the ordered vertex set $V_x = \{\sigma, \rho, 2\rho, 2^2\rho, \cdots, 2^n \rho\}$ in order. In order for there to be an edge between the vertices $\sigma$ and $\rho$, $\tau_0$ must divide $x$. In order for there to be an edge between the vertices $2^{k} \rho$ and $2^{k + 1} \rho$, $\xi_1^{{2^k}}$ must divide $x$. Hence 
\[
\tau_0 \xi_1^{\sum_{k = 0}^n 2^k} = \tau_0 \xi_1^{2^{n + 1}-1}
\]
must divide $x$. On the other hand, if $\tau_0 \xi_1^{2^{n + 1}-1}$ divides $x$, then the graph $G_x^\textup{dir}$ has a Hamiltonian-directed path by construction.
\end{proof}

Next we will turn to establishing a criteria for when the graph $G_x$ has a Hamilton cycle. Our goal will be to again make use of Dirac's theorem, which says that a simple graph with $n$ vertices (${\displaystyle n\geq 3}$) is Hamiltonian if every vertex has degree ${\displaystyle {\tfrac {n}{2}}}$ or greater. Unlike the nonequivariant odd prime setting, the graphs in the $C_2$-equivariant setting are all simple graph so the arguments here are simpler than those of \cref{sec:Trees&HamiltonOdd}. 

Let $A_x$ be the adjacency matrix associated to $G_x$ with the $(i,j)$ entry of $A_x$ denoted $a_{i,j}$.
Define
    \[
    \deg_{\out} (2^j \rho) = \#\{1 < i \leq n | a_{j + 2, j + 2 + i} = 1\}
    \]
    and 
    \[
    \deg_{\inn} (2^j\rho) = \#\{1 < i \leq j+1 | a_{j + 2 -i, j + 2} = 1\}
    \] 
The following proposition is immediate from the definitions of $A_x$ and $G_x$.

\begin{proposition}  
    The out-degree of the vertex $2^j \rho$ in $G_x^\textup{dir}$ is given by $\deg_{\out}$. The in-degree of the vertex $2^j \rho$ in $G_x^\textup{dir}$ is given by $\deg_{\inn} (2^j\rho)$.
\end{proposition}

Since $\deg_{\inn}(2^j\rho) + \deg_{\out}(2^j\rho)$ gives the total degree of the vertex $2^j \rho$ ($0 \leq j \leq n$) in $G_x^\textup{dir}$, it only remains to consider the degree of the vertex $\sigma$ in order to apply Dirac's theorem. 

Since $\sigma$ is the initial vertex in the directed graph $G_x^\textup{dir}$, the total degree of the vertex $\sigma$ is 
\[
\sum_{k = 0}^n \epsilon_k.
\]
Then using Dirac's theorem we arrive at the following result.
\begin{theorem} \label{thm:HamiltonCycleC2}
    Suppose $x = \tau_0^{\epsilon_0} \cdots \tau_n^{\epsilon_n} \xi_1^{r_1} \cdots \xi_n^{r_n}$ is a monic monomial { in standard form} in $ A_{C_2}^\vee(n)$. The corresponding graph $G_x$ has a Hamilton cycle if
    \begin{enumerate}
        \item $n>0$,
        \item for each $j$ with $0 \leq j \leq n$,
        \[
            \deg_{\inn}(2^j \rho) + \deg_{\out} (2^j \rho) \geq         \frac{n}{2}, 
        \]
        \item and \[\sum_{k=0}^n  \epsilon_k \geq \frac{n}{2}.\]
    \end{enumerate}  
\end{theorem}

\subsection{Graph theoretic interpretation of coproduct and antipode}

Similarly to Larson's graph theoretic description of the coproduct at the prime $p =2$ (\cref{thm:LarsonCoproduct}) and our graph theoretic description at odd primes (\cref{sec:coprodOdd}), we also have a graphical interpretation of the coproduct on elements $\xi_i^{2^j}$ and $\tau_i \in A_{C_2}^\vee (n)$. 

\begin{theorem} \label{thm:coprodC2} 
     The coproduct $\psi(\xi_i^{p^j}) \in A^\vee_{C_2}(n)$ is the sum of tensors of all pairs of edges that make length $2$ directed paths from the vertex $2^j \rho $ to the vertex $2^{i + j} \rho$ in the complete graph on the vertices $\{\sigma, \rho, 2 \rho, 2^2 \rho, \cdots 2^n \rho\}$ considered as a directed graph.
     
    Similarly, the coproduct $\psi(\tau_i) \in A_{C_2}^\vee(n)$ is the sum of tensors of all pairs of edges that make length $2$ directed paths from the vertex $\sigma$ to the vertex $2^i \rho$. \color{black}
\end{theorem}

\begin{proof}
Recall that the coproduct on $A_{C_2}^\vee(n)$ is given by

\begin{align*}
    \psi(\xi_n) & =\sum_{i=0}^n \xi_{n-i}^{2^i} \otimes \xi_i \\ 
    \psi (\tau_n) & = \tau_n \otimes 1 + \sum_{i=0}^n \xi_{n-i}^{2^i} \otimes \tau_i.
\end{align*}
This formulas have the same form as in the case of the odd primary nonequivariant dual Steenrod algebra and the proof is the same as in that case.
\end{proof}

Using a similar argument as in the nonequivariant $p$ an odd prime case, we obtain a graphical description of the antipode on elements $\xi_i^{2^j}$ and $\tau_i \in A_{C_2}^\vee (n)$.

\begin{theorem} \label{thm:AntipodeC2} 
    The antipode $c(\xi_i^{p^j}) \in A^\vee_{C_2}(n)$ is the sum of all directed paths from the vertex $2^j \rho$ to the vertex $2^{i + j} \rho$ in the complete graph on the vertices $\{\sigma, \rho, 2 \rho, 2^2 \rho, \cdots 2^n \rho\}$ considered as a directed graph.

    Similarly, the antipode $c(\tau_i) \in A^\vee_{C_2}(n)$ is the sum of all directed paths from the vertex $\sigma$ to the vertex $2^i \rho$. \color{black}
\end{theorem}

As a corollary, we also obtain another characterization of the unilaterality of the directed graphs $G_x^\textup{dir}$.

\begin{cor} \label{cor:unilatC2}
    Suppose $x$ is a monic monomial { in standard form} in $A_{C_2}^\vee (n)$. Then the directed graph $G_x^{\textup{dir}}$ is unilateral if and only if
    \begin{enumerate}
        \item for each $\xi_i^{2^j} \in A_{C_2}^\vee (n)$, at least one summand of $c(\xi_i^{2^i})$ is a factor of $x$
        \item and $\tau_0$ is a factor of $x$.
    \end{enumerate}
\end{cor}

\printbibliography

\end{document}